\documentclass[12pt]{article}
\usepackage[ 
margin = 3cm, bottom = 4cm
]{geometry}
\usepackage{setspace}

\usepackage[T5]{fontenc}
\usepackage[utf8]{inputenc}
\usepackage[english
]
{babel}
\usepackage{url}

\usepackage[backend=biber
, style=alphabetic
]{biblatex}
\usepackage{amsfonts}
\usepackage{amsmath}
\usepackage{amssymb}
\usepackage{amsthm}
\usepackage{cancel}
\usepackage{enumitem}
\usepackage{etoolbox}
\usepackage{expl3}
\usepackage{mathdots}
\usepackage{mathtools}
\usepackage{pgfkeys}
\usepackage{pgfopts}
\usepackage{rotating}
\usepackage{stmaryrd}
\usepackage{tikz}
\usepackage{tikz-cd}
\usepackage{xparse}
\usepackage{xstring}

\usepackage[mathscr]{euscript}
\usepackage{newtxmath}
\usepackage{times}

\usepackage[
hyperindex,linktocpage=true]{hyperref}
\hypersetup{
    colorlinks,
    linkcolor={red!60!black},
    citecolor={blue!50!black},
    urlcolor={blue!80!black}
}

\theoremstyle{plain}
\newtheorem{THM}{Theorem}[section]

 \newtheorem{Lemma}[THM]{Lemma}
 \newtheorem{Prop}[THM]{Proposition}

\theoremstyle{definition}
 \newtheorem{DEF}[THM]{Definition}
 \newtheorem{Eg}[THM]{Example}
 
\DeclarePairedDelimiter\parens{\lparen}{\rparen}
\newcommand{\llaurent}{(\!(}
\newcommand{\rlaurent}{)\!)}
\DeclarePairedDelimiter\laurents{\llaurent}{\rlaurent}
\DeclarePairedDelimiter\bbrackets{\llbracket}{\rrbracket}

\newcommand{\Alg}[1]{#1\text{-}\mathrm{Alg}}

\newcommand{\Aut}[1]{\operatorname{Aut}{#1}}

\newcommand{\chp}[1]{\operatorname{ch}{\parens{#1}}}

\newcommand{\Conv}[1]{\operatorname{Conv}{\parens{#1}}}

\newcommand{\Hom}[2]{\operatorname{Hom}{\parens{#1, #2}}}
\newcommand{\HomA}[3]{\operatorname{Hom}_{#1}{\parens{#2,#3}}}
\newcommand{\id}[1]{\operatorname{id}_{#1}}

\newcommand{\Irr}[1]{\operatorname{Irr}{\parens{#1}}}

\newcommand{\res}[2]{\operatorname{res}_{#1}{\parens{#2}}}
\newcommand{\Sets}{\operatorname{Sets}}

\newcommand{\Spec}[1]{\operatorname{Spec}{#1}}

\newcommand{\tr}[1]{\operatorname{tr}{#1}}
\newcommand{\trres}[2]{\operatorname{tr}{\parens{#1 \mid #2}}}

\newcommand{\valp}[2]{\operatorname{val}_{#1}{\parens{#2}}}

\newcommand{\dimp}[1]{\operatorname{dim}{\parens{#1}}}

\newcommand{\ModA}[1]{\mathrm{Mod}_{#1}}

\newcommand{\BB}{\mathbb{B}}
\newcommand{\CC}{\mathbb{C}}

\newcommand{\GG}{\mathbb{G}}
\newcommand{\HH}{\mathbb{H}}

\newcommand{\RR}{\mathbb{R}}

\newcommand{\ZZ}{\mathbb{Z}}

\newcommand{\Aa}{\mathscr{A}}

\newcommand{\Ff}{\mathscr{F}}
\newcommand{\Gg}{\mathscr{G}}

\newcommand{\Oo}{\mathscr{O}}

\newcommand{\frakg}{\mathfrak{g}}
\newcommand{\frakh}{\mathfrak{h}}

\newcommand{\Gr}{\mathrm{Gr}}

\newcommand{\GrG}[1]{\Gr_{#1}}

\newcommand{\Rep}[2]{\mathrm{Rep}_{#1}\parens{#2}}
\newcommand{\Vect}[1]{\mathrm{Vec}_{#1}}

\DeclareSymbolFont{LMletters}{OML}{lmm}{m}{it}
\DeclareMathSymbol{\conv}{\mathbin}{LMletters}{63}

\title{The twining character formula for reductive groups}
\author{Jackson Hopper}

\begin{document}
 \maketitle
 
 \begin{abstract}
 Let $\widehat{G}$ be a connected reductive group over an algebraically closed field with a pinning-preserving outer automorphism $\sigma$.  Jantzen's twining character formula relates the trace of the action of $\sigma$ on a highest-weight representation $V_{\mu}$ of $\widehat{G}$ to the character of a corresponding highest-weight representation $(V_{\sigma})_{\mu}$ of a related group $\widehat{G^{\sigma, \circ}}$. This paper extends the methods of Hong's geometric proof for the case $\widehat{G}$ is adjoint, to prove that the formula holds for all connected reductive groups, and examines the role of additional hypotheses. In the final section, it is explained how these results can be used to draw conclusions about quasi-split groups over a non-Archimedean local field. This paper thus provides a more general geometric proof of the Jantzen twining character formula and provides some apparently new results of independent interest along the way.
 \end{abstract}
 
 \tableofcontents
 
 \bigskip
 
%
%
 
 Jantzen's twining character formula is a twisted version of the Weyl character formula. Given a pinning-preserving outer automorphism $\sigma$ on a connected, reductive group scheme $\widehat{G}$ over an algebraically closed field, the formula describes the twisted character of $\sigma$ on a highest-weight representation of $\widehat{G}$ in purely combinatorial terms, and can be calculated using the $\sigma$-action on the root datum of $\widehat{G}$. It was first proved by Jantzen, \cite{Ja73}, with alternative proofs provided by \cite{KLP09}, \cite{Ho09}, and \cite{CCH19}.
 
 Most of the above proofs share an assumption that the group $\widehat{G}$ is connected, semisimple, and adjoint, and some impose additional hypotheses. However, the proof of \cite{CCH19} holds for connected, reductive groups, following some cohomology calculations in \cite{Kos61}. In this paper I will make the same assumptions: that $\widehat{G}$ is a connected, and reductive group.
 
 I will follow quite closely the geometric proof of Hong \cite{Ho09}. I will outline the structure, then reproduce the proof in a fairly self-contained way to make clear where the stronger hypotheses might be convenient---and why they are unnecessary. For one thing, Hong's proof applies equally well as written whether $\widehat{G}$ is assumed to be adjoint or simply connected. For the more general case, Proposition \ref{Prop:no-circ} will be useful. Although Proposition \ref{Prop:no-circ} is not necessary to prove Theorem \ref{THM:main}, as shown in \cite{CCH19}, I believe it is interesting in its own right and have not seen it elsewhere in the literature.
 
 \bigskip
 \bigskip
 
 Let $\widehat{G}$ be a connected, reductive group over an algebraically closed field $K$ of characteristic $0$, and fix a root datum of $\widehat{G}$. In particular, fix a maximal torus and Borel $\widehat{T} \subset \widehat{B} \subset \widehat{G}$. Let $G$ be the complex group with dual root datum, and with a corresponding choice of maximal torus and Borel $T \subset B \subset G$. Let $\sigma$ be an automorphism of $\widehat{G}$ preserving the root datum and a pinning, and consider its induced action on $G$, which preserves the dual root datum and a pinning. Let $G^{\sigma}$ be the fixed-point subgroup of $G$, and let $G^{\sigma, \circ}$ be the neutral component of $G^{\sigma}$. Then $G^{\sigma, \circ}$ is a connected, reductive group \cite{St68} (see also \cite{Ha15}), and a closed subgroup of $G$, with maximal torus and Borel $T^{\sigma, \circ} \subset B^{\sigma, \circ} \subset G^{\sigma, \circ}$. The cocharacter lattice $X_*(T^{\sigma, \circ})$ is a subgroup of $X_*(T)$, and other components of the root datum of $G^{\sigma, \circ}$ can also be determined combinatorially. Let $\widehat{G^{\sigma, \circ}}$ be the dual of $G^{\sigma, \circ}$ over $K$.
 
 Note that $\sigma$ acts on the lattice of cocharacters of $T$ (i.e. characters of $\widehat{T}$). If $\mu$ is any $\sigma$-invariant dominant cocharacter of $G$, then we are interested in the action of $\sigma$ on the irreducible highest-weight representation $V_{\mu}$ of $\widehat{G}$, as well as the action of $\sigma$ on weight spaces $V_{\mu}(\lambda)$, where $\lambda$ is a nonzero weight of $V_{\mu}$. The set of such weights is denoted $Wt(\mu)$. Up to a scalar, there is a unique vector space automorphism $\sigma: V_{\mu} \to V_{\mu}$ commuting with the action of $\widehat{G}$ on $V_{\mu}$. We can normalize this automorphism by assuming $\sigma$ fixes the highest-weight line $V_{\mu}(\mu)$, uniquely determining an action of $\sigma$. The irreducible representation of $\widehat{G^{\sigma, \circ}}$ of highest weight $\mu$ is denoted $(V_{\sigma})_{\mu}$. Then we have the following theorem relating $V_{\mu}$ and $(V_{\sigma})_{\mu}$.

 \begin{THM}[Jantzen's twining character formula]
  \label{THM:main} 
  Let $\widehat{G}$, $G$, and $\sigma$ be as above. Let $\mu$ be a $\sigma$-invariant dominant character in $X^*(\widehat{T})$ and let $\lambda \in Wt(\mu)$ be a $\sigma$-invariant weight of $V_{\mu}$. Then $\sigma$ preserves $V_{\mu}(\lambda)$, and we have the following equality:
  \begin{equation}
   \label{eq:main}
   \trres{\sigma}{V_{\mu}(\lambda)} = 
   \dim{((V_{\sigma})_{\mu}(\lambda))}.
  \end{equation}
  
  The Weyl character formula for $\widehat{G^{\sigma, \circ}}$ thus implies a twining formula for the twisted character of $\sigma$:
  \[
   \sum_{\substack{\lambda \in Wt(\mu)\\ \sigma(\lambda) = \lambda}} \trres{\sigma}{V_{\mu}(\lambda)} e^{\lambda} =
   \chp{(V_{\sigma})_{\mu}} =
   \sum_{w \in W^{\sigma}} w \left(\prod_{\alpha \in N'_{\sigma} (\Phi)^+} \frac{1}{1 - e^{-\alpha}} \right) e^{w(\mu)}.
  \]
  Here $N'_{\sigma} (\Phi)$ is a root system explicitly determined by the $\sigma$-action on $\Phi$ and is the root system of the group $\widehat{G^{\sigma, \circ}}$.
 \end{THM}
 
 See Section \ref{S:Haines} for details on how to determine $N'_{\sigma} (\Phi)$ using the root datum of $\widehat{G}$. Notation and proofs there are drawn from \cite{Ha18}.
 
 There are two apparent justifications for the stronger hypotheses taken in previous proofs of Theorem \ref{THM:main}. First, unless $\widehat{G}$ is semisimple and either simply connected or adjoint, $G^{\sigma}$ may not be connected. This turns out to be immaterial, as the affine Grassmannian of $G^{\sigma}$ (and the category of sheaves on it) ``forgets'' any disconnectedness of $G^{\sigma}$, cf. Proposition \ref{Prop:no-circ}.
 
 Second, the root lattice of $\widehat{G}$ is a strict sublattice of $X^*(\widehat{T})$ in the case $\widehat{G}$ is not semisimple and adjoint. In this case the affine Grassmannian of $G$ is disconnected, and the dimensions of Schubert varieties vary in differing components. However, it turns out that, due to the use of Anderson's polytope calculus and normalization to stable AMV cycles, disconnectedness of the affine Grassmannian is also immaterial.
 
 In Section \ref{S:Notation}, I establish conventions used throughout the paper. In Section \ref{S:Hong}, I outline the proof of Theorem \ref{THM:main}, omitting some details, to see how the stronger hypotheses are used in the literature. Sections \ref{S:Orbits} through \ref{S:Eigenvalues} comprise a complete proof of Theorem \ref{THM:main}: Section \ref{S:Orbits} describes the $\sigma$-action on subvarieties of the affine Grassmannian of $G$, relating them to corresponding varieties in the affine Grassmannian of $\GrG{G^{\sigma}}$; Section \ref{S:Lusztig} considers the action of $\sigma$ on $\mathbf{i}$-Lusztig strata, establishing a condition for invariance; Section \ref{S:Gathering} establishes a coweight-preserving bijection between the $\sigma$-invariant MV cycles of $G$ and all MV cycles of $G^{\sigma, \circ}$; and Section \ref{S:Eigenvalues} completes the proof of Theorem \ref{THM:main} by showing that $\sigma$ fixes basis vectors corresponding to $\sigma$-invariant MV cycles, implying that the trace of $\sigma$ is exactly the number of preserved MV cycles. Section \ref{S:Haines} deals explicitly with root data and uses Theorem \ref{THM:main} to prove a Theorem 7.7, which stated in \cite{Ha18} without reference to a fully general proof. Finally, Appendix \ref{app} is a complete list of results from \cite{MV07} used in this paper.
 
 \section{Notation}
  \label{S:Notation}
  Here I will establish some notation and conventions. Throughout, $\widehat{G}$ is a connected, reductive group over an algebraically closed field $K$ of characteristic $0$. I primarily work with its complex dual group $G$. If I write ``character,'' ``coweight,'' or other similar term, without specifying which group I am referring to, I intend to refer to $G$.
   
  Fix a maximal torus and Borel $T \subset B \subset G$. We also have the corresponding maximal unipotent subgroup $U \subset B$. Let the set of simple roots be denoted $\Pi$ and the simple coroots denoted $\Pi^{\vee}$. Let the set of roots be denoted $\Phi$ and the set of coroots denoted $\Phi^{\vee}$. Let the character lattice be denoted $X^*(T)$ and the cocharacter lattice denoted $X_*(T)$.
  
  Fix also a maximal torus and Borel $\widehat{T} \subset \widehat{B} \subset \widehat{G}$, with corresponding maximal unipotent $\widehat{U}$. Then the character lattice for $\widehat{G}$ is exactly $X^*(\widehat{T}) = X_*(T)$, and the cocharacter lattice is $X_*(\widehat{T}) = X^*(T)$; the set of roots of $\widehat{G}$ is $\Phi^{\vee}$, the set of coroots is $\Phi$; the set of simple roots of $\widehat{G}$ is $\Pi^{\vee}$, and the set of simple coroots is $\Pi$.
   
  Let $G^{der}$ be the derived subgroup of $G$, and $G^{sc}$ the universal cover of $G^{der}$. Then $G^{sc}$ and $G^{der}$ have the same set of coroots and simple coroots as $G$, and there is a natural embedding $X_*(T^{sc}) \hookrightarrow X_*(T)$, with $X_*(T^{sc})$ generated by $\Pi^{\vee}$. Given any two cocharacters $\mu, \lambda \in X_*(T)$, we say $\lambda \le \mu$ if and only if $\mu - \lambda \in \ZZ_{\ge 0} \Pi^{\vee}$. Let $\rho$ be the half sum of positive roots of $G$ and $\rho^{\vee}$ be the half sum of positive coroots.
   
  Let $W$ be the Weyl group $N_G(T)/T$. Then $W$ acts on $X^*(T)$ and $X_*(T)$. For each $w \in W$, let $\lambda \le_w \mu$ if and only if $w^{-1}(\lambda) \le w^{-1}(\mu)$. Corresponding to the choice of simple roots $\Pi$, we have a set of simple reflections $S = \{s_{\alpha}\}_{\alpha \in \Pi}$ generating $W$, and $(W, S)$ is a Coxeter system. Then there is a length function $\ell$ on elements of $W$; let $w_0 \in W$ be the longest element.
  
  Fix a pinning of $G$ compatible with $T$ and $B$, i.e. a collection of root homomorphisms $x_{\alpha}: \GG_a \to U$ for each simple root $\alpha \in \Pi$. Each root homomorphism $x_{\alpha}$ also uniquely determines an opposite root homomorphism $y_{\alpha}: \GG_a \to w_0 U w_0^{-1}$. Fix also pinning $\{x_{\alpha^{\vee}}\}_{\alpha^{\vee} \in \Pi^{\vee}}$ of $\widehat{G}$ compatible with $\widehat{T}$ and $\widehat{B}$.
   
  Let $\sigma$ be an automorphism of $G$ preserving the pinning $\{x_{\alpha}\}_{\alpha \in \Pi}$, meaning that $\sigma$ preserves $T$ and $B$, and that $\sigma \circ x_{\alpha} = x_{\sigma(\alpha)}$ for all $\alpha \in \Pi$. Let $G^{\sigma}$ be the fixed point subgroup, and $G^{\sigma, \circ}$ the neutral component of that fixed point subgroup.
  
  Since $\sigma$ acts on the constituents of the root datum of $\widehat{G}$, there is a unique action of $\sigma$ on $\widehat{G}$ preserving its root datum and the pinning $\{x_{\alpha^{\vee}}\}$. Specifically, $\widehat{G}$ is generated by the images of cocharacters generating $X^*(T)$ and by the root and opposite root homomorphisms $x_{\alpha^{\vee}}$ and $y_{\alpha^{\vee}}$ for $\alpha^{\vee} \in \Pi^{\vee}$. It is thus sufficient to define $\sigma$ on these images. Let $R$ be a $K$-algebra, and suppose $g \in \widehat{G}(R)$. If $g = \lambda(t)$ for some $t \in \GG_a(R)$ and $\lambda \in X^*(T)$, let $\sigma(g) = \sigma(\lambda)(t)$. And if $g = x_{\alpha^{\vee}}(u)$ for some $u \in \GG_a(R)$, then let $\sigma(g) = x_{\sigma(\alpha^{\vee})}(u)$; similarly, if $g=y_{\alpha^{\vee}}(u)$, let $\sigma(g) = y_{\sigma(\alpha^{\vee})}(u)$.
   
  For a complex, smooth, linear algebraic group $H$ we have the loop group, positive loop group, negative loop group, and strictly negative loop group functors from $\CC$-algebras to sets given by $LH: R \mapsto H(R\laurents{\varpi})$, $L^+H: R \mapsto H(R\bbrackets{\varpi})$, $L^-H: R \mapsto H(R[\varpi^{-1}]) \subset LH(R)$, and $L^{--}H: R \mapsto \ker{(L^-H(R) \to H(R))}$, respectively. The \'etale sheafification of the quotient functor $LH/L^+H: R \mapsto H(R\laurents{\varpi})/H(R\bbrackets{\varpi})$ is known as $\GrG{H}$, the affine Grassmannian of $H$, and is representable by an ind-finite type (strict) ind-scheme over $\CC$. The ind-scheme $\GrG{H}$ is ind-projective if and only if $H$ is reductive (see, for instance, \cite{BD91} Theorem 4.5.1(iv)). For this reason, it is essential to this proof to assume $G$ is reductive.
  
  For each cocharacter $\nu \in X_*(T)$ we have by definition a homomorphism $\nu: \GG_m \to T \subset G$, as well as a homomorphism $\nu: L\GG_m \to LT$. Let $\varpi^{\nu} \in LT(\CC)$ be the image of $\varpi$ under this homomorphism, and let $\varpi^{\nu} x_0$ be the image of $\varpi$ under the composition 
  \[
   L\GG_m (\CC) \overset{\nu}{\to} 
   LT(\CC) \to
   LG(\CC) \to 
   \GrG{G}(\CC),
  \]
  where $x_0$ is the natural basepoint of $\GrG{G}(\CC)$, corresponding to the trivial coset in $LG(\CC)/L^+G(\CC)$. 
   
  Given a locally closed, reduced sub-ind-scheme $Y \subset \GrG{G}$, let $\overline{Y}$ be the reduced closure. If $G$ is reduced, then $\overline{Y}$ is ind-projective. $\GrG{G}$ has a Cartan stratification by $L^+G$-orbits. Given a cocharacter $\mu$, let $\GrG{G}^{\mu}$ be the $L^+G$-orbit $\GrG{G}^{\mu} = L^+G \cdot \varpi^{\mu}x_0$. I refer to these orbits as Schubert cells, and their closures as Schubert varieties. Schubert cells and Schubert varieties are reduced, finite-type, complex schemes. Typically $\mu$ will be taken dominant, since $\GrG{G}^{\mu} = \GrG{G}^{w(\mu)}$ for all $w \in W$. If $\mu$ is dominant, we have the following closure relations from the Cartan stratification:
  \[
   \overline{\GrG{G}^{\mu}} =
   \coprod_{\substack{\lambda \in X_*(T)^+ \\ \lambda \le \mu}} \GrG{G}^{\lambda}.
  \]
  We also have, for each $w \in W$, an Iwasawa stratification. The strata of the Iwasawa stratification are known as semi-infinite cells. In contrast with Schubert varieties, semi-infinite cells and their closures are not representable by schemes. Given $w$ and a cocharacter $\nu$, let $S_w^{\nu}$ be the orbit $S_w^{\nu} = wLUw^{-1} \cdot \varpi^{\nu}x_0$. Then we have the following closure relations (see eg \cite{MV07} Proposition 3.1(a)):
  \[
   \overline{S_w^{\nu}} =
   \coprod_{\substack{\eta \in X_*(T) \\ \eta \le_w \nu}} S_w^{\eta}
  \]
  
  From a geometric description of the complex points, we have an intersection criterion (see eg \cite{MV07} equation (3.5) in the proof of Theorem 3.2). That geometric description is
  \begin{equation}
   \label{eq:semi-infinite}
   S_w^{\nu}(\CC) =
   \{x \in \GrG{G}(\CC) \mid \lim_{s \to 0} w(\rho^{\vee}) (s) \cdot x = \varpi^{\nu} x_0\},
  \end{equation}
  where $w(\rho^{\vee}): \GG_m(\CC) \to T(\CC)$ is a homomorphism of complex groups. As a consequence of this description, $S_w^{\eta} \cap S_{w'}^{\nu} \ne \varnothing$ only if $\nu \le_w \eta$ and $\eta \le_{w'} \nu$. Indeed, if $p \in (S_w^{\eta} \cap S_{w'}^{\nu}) (\CC)$, then both $\varpi^{\eta} x_0$ and $\varpi^{\nu} x_0$ are in the closure of the $T(\CC)$-orbit of $p$. Since both $S_w^{\eta}$ and $S_{w'}^{\nu}$ are $T$-invariant, that means in particular that $\varpi^{\eta} x_0 \in \overline{S_{w'}^{\nu}}(\CC)$ and $\varpi^{\nu} x_0 \in \overline{S_w^{\eta}} (\CC)$, implying the inequalities.
  
  Given a reduced, irreducible, projective subvariety $X \subset \GrG{G}$, the sheaf $IC_X = j_{!*} (\CC[\dim{X}])$ is the unique perverse sheaf restricting to constant coefficients on the nonsingular locus of $X$. In the case $X = \overline{\GrG{G}^{\mu}}$ for a dominant cocharacter $\mu$, this sheaf is $L^+G$-equivariant and known simply as $IC_{\mu}$. The category of $L^+G$-equivariant perverse sheaves on closed subvarieties of $\GrG{G}$ consists of only direct sums of $IC_{\mu}$ for dominant $\mu$ and is referred to as $P_{L^+G} (\GrG{G})$.
 
 \section{Outline of proof}
  \label{S:Hong}
  Here I will summarize the proof, adapted from \cite{Ho09}, of Theorem \ref{THM:main}. I do this primarily to see that the hypothesis that $G$ is reductive sufficient. Suppose $G$ is a complex, semisimple, simply connected group, and let $\sigma$ be a pinning-preserving automorphism of $G$.
  
  The proof is geometric in nature, relying on the geometry of the affine Grassmannian $\GrG{G}$. Of particular importance is the geometric Satake equivalence, which constructs an explicit and canonical bijection between certain varieties contained in $\GrG{G}$, called MV cycles, and basis vectors of highest-weight representations of $\widehat{G}$. For precise statements of the several theorems from \cite{MV07} I am referring to when I say ``the geometric Satake equivalence,'' see the Appendix \ref{app}. The most important results, stated according to conventions from Section \ref{S:Notation}, are summarized here:
  
  \begin{THM}
    \label{THM:MV}
    Let $\mu$ be a dominant cocharacter, and let $\lambda \in Wt(\mu)$.
    \begin{enumerate}[label=\roman*.]
     \item
      \label{THM:MV:i}
      $S_{w_0}^{\lambda} \cap \overline{\GrG{G}^{\mu}}$ is equidimensional, and $\dimp{S_{w_0}^{\lambda} \cap \overline{\GrG{G}^{\mu}}} = \langle \rho, \mu - \lambda \rangle$
      
     \item
      \label{THM:MV:ii}
      $S_{w_0}^{\lambda} \cap S_e^{\mu}$ is equidimensional, and $\dimp{S_{w_0}^{\lambda} \cap S_e^{\mu}} = \langle \rho, \mu - \lambda \rangle$
     
     \item
      \label{THM:MV:iii}
      $\displaystyle \HH^{\bullet}(\GrG{G}, IC_{\mu}) = \bigoplus_{\lambda \in Wt(\mu)} H_c^{-2\langle \rho, \lambda\rangle} (S_{w_0}^{\lambda}, IC_{\mu}) =
       V_{\mu}
       $
      
     \item
      \label{THM:MV:iv}
      $\displaystyle H_c^{-2\langle \rho, \lambda\rangle}(S_{w_0}^{\lambda}, IC_{\mu}) =
       \bigoplus_{A \in \Irr{S_{w_0}^{\lambda} \cap \overline{\GrG{G}^{\mu}}}} K[A] =
       V_{\mu}(\lambda).$
       
      \item
       \label{THM:MV:v}
       $P_{L^+G} (\GrG{G}, \ZZ)$ is isomorphic as a tensor category to $\Rep{\ZZ}{\widehat{G}}$. 
     \end{enumerate}
   \end{THM}
   
   The direct sum in statement \ref{THM:MV:iv} is indexed by irreducible components $A$ of the variety $S_{w_0}^{\lambda} \cap \overline{\GrG{G}^{\mu}}$. These irreducible components are the MV cycles. The heart of the proof of Theorem \ref{THM:main} is to establish a bijective correspondence between those MV cycles in $\GrG{G}$ which are invariant under the action of $\sigma$ and all MV cycles in $\GrG{G^{\sigma}}$.
   
   This is done in three steps. First, the collection of MV cycles is generalized to a larger collection of what are here called AMV cycles (see Section \ref{S:Orbits} for definition), after Anderson's polytope calculus \cite{An03}. Using Kamnitzer's indexing of AMV cycles by $\mathbf{i}$-Lusztig data (see Section \ref{S:Lusztig}), a convenient criterion for $\sigma$-invariance of an AMV cycle is found, as well as a procedure for finding the datum of a corresponding AMV cycle in $\GrG{G^{\sigma}}$. Second, a criterion of \cite{An03} for when an AMV cycle intersects generically with an MV cycle is applied to show that the restriction of this correspondence to MV cycles is also bijective. Finally, the eigenvalues of the $\sigma$-action on the $\widehat{G}$-representation $V_{\mu}$ are examined to ensure that the twisted character of $\sigma$ is exactly as expected, completing the proof of Theorem \ref{THM:main}.
   
   Based on the short description above, it is not obvious how the hypothesis that $G$ is semisimple is used. However, much of the literature explicitly makes this and other assumptions. 
   
   The two complications introduced by relaxing the hypotheses from semisimple and simply connected to reductive come in the form of two different disconnected spaces. First, if $G$ is not either semisimple or adjoint, then the fixed point subgroup $G^{\sigma}$ is not necessarily connected. Using the classification of reductive groups, it is more convenient to work with $G^{\sigma, \circ}$ than $G^{\sigma}$ whenever dealing with root data. However, it is easier, and in my opinion more natural, to relate the geometry of $(\GrG{G})^{\sigma}$ to $\GrG{G^{\sigma}}$ than to $\GrG{G^{\sigma, \circ}}$. Thankfully, the affine Grassmannians of $G^{\sigma}$ and $G^{\sigma, \circ}$ are isomorphic for all reductive $G$ (see Proposition \ref{Prop:no-circ}). So results including Proposition \ref{Prop:Orbits} and its consequences still go through without issue, as they may be applied to $\GrG{G^{\sigma}}$. But the group whose category of representations is isomorphic to $P_{L^+G^{\sigma}}(\GrG{G^{\sigma}})$ is $\widehat{G^{\sigma, \circ}}$, as needed for Theorem \ref{THM:main}.

   The second complication is that $\pi_0(\GrG{G}) = \pi_1(G)$. So if $G$ is not semisimple and simply connected, then $\GrG{G}$ is not connected. However, in the framework of AMV cycles introduced by \cite{An03}, this is not a complication at all. In fact, both Kamnitzer and Hong work  primarily with stable AMV cycles, which are AMV cycles translated by $X_*(T)$ to be contained in $\overline{S_e^0}$, which is itself contained in the neutral component of $\GrG{G}$. This does not rule out possible sticking points in passing between MV cycles, AMV cycles, stable AMV cycles and back to MV cycles, but every result solely relating to stable AMV cycles holds automatically for all reductive groups.
   
   And the result necessary to restrict the bijection on the level of AMV cycles to MV cycles is Theorem \ref{THM:Polytope}, which is proved again without a need for triviality $\pi_1(G)$. More details are presented in Section \ref{S:Orbits}.
   
   In short, considering reductive groups, rather than semisimple groups---much less simply connected groups---is no more complicated for the proof of Theorem \ref{THM:main}. Every potential obstacle is either immaterial or easily deflected. In particular, Proposition \ref{Prop:no-circ} and the preceding lemmas are the only results I had not previously found in the literature.

  \section{Action of $\sigma$ on orbits}
   \label{S:Orbits}
   For the remainder of the paper, suppose $G$ is a connected, reductive, complex group, and $\sigma$ is a pinning-preserving automorphism on $G$.
   
   Unlike in the case $G$ is semisimple and simply connected, we cannot count on $G^{\sigma}$ to be a connected group in general. This leads us to a choice: should we work with $G^{\sigma, \circ}$, to which the classification of connected, reductive groups applies, or should we work directly with $G^{\sigma}$, which has a simpler description relative to $G$? Thanks to the upcoming Proposition \ref{Prop:no-circ}, it is immaterial whether we work with $G^{\sigma}$ or $G^{\sigma,\circ}$. I work primarily with $G^{\sigma}$ for simplicity, and the final result will hold for $G^{\sigma, \circ}$. I will avoid referring to root datum of $G^{\sigma}$ where possible. Note, however, that $\Hom{\GG_m}{T^{\sigma}} = \Hom{\GG_m}{T^{\sigma, \circ}}$, so $X_*(T^{\sigma}) = X_*(T^{\sigma, \circ})$.
   
   First, we will need two lemmas relating loop groups of algebraic groups and their quotients.
   
   \begin{Lemma}
    \label{Lemma:quotient-L-pos}
    Let $G$ be an affine group scheme over a field $k$, and suppose $H \subset G$ is a smooth normal subgroup with affine quotient $G/H$. There is a natural isomorphism of functors $L^+G/L^+H \to L^+(G/H)$, where $L^+G/L^+H$ is the \'etale quotient.
   \end{Lemma}
   
   In particular, if $H \subset G$ is a normal subgroup and both groups are reductive, then $L^+G/L^+H$ and $L^+(G/H)$ are canonically isomorphic group schemes.
   
   \begin{proof}
    Let the quotient map be denoted $q_0: G \to G/H$. Note that $L^+G = \varprojlim G^{(n)}$, where $G^{(n)}$ is the $n$th jet group $R \mapsto G(R[\varpi]/(\varpi^{n+1}))$. Similarly, $L^+(G/H) = \varprojlim (G/H)^{(n)}$. Hence the map $q: L^+G \to L^+(G/H)$ corresponds to the inverse system of morphisms
    \[
     \begin{tikzcd}
      \cdots \ar[r] &
      G^{(n)} \ar[r, "i_n"] \ar[d, "q_n"] &
      G^{(n-1)} \ar[r, "i_{n-1}"] \ar[d, "q_{n-1}"] &
      \cdots \ar[r] &
      G^{(1)} \ar[r, "i_1"] \ar[d, "q_1"] &
      G \ar[d, "q_0"]
       \\
      \cdots \ar[r] &
      (G/H)^{(n)} \ar[r, "j_n"] &
      (G/H)^{(n-1)} \ar[r, "j_{n-1}"] &
      \cdots \ar[r] &
      (G/H)^{(1)} \ar[r, "j_1"] &
      G/H
     \end{tikzcd}
    \]
    
    The proof of surjectivity will proceed inductively. For each $n \ge 1$, I will use surjectivity and formal smoothness of $q_{n-1}$ to prove that $q_n$ is surjective.
    
    Suppose $q_{n-1}$ is surjective, and let $g_n \in (G/H)^{(n)} (R)$, for a $k$-algebra $R$, with image $g_{n-1} \in (G/H)^{(n-1)} (R)$. By surjectivity of $q_{n-1}$, there is a lift $\tilde{g}_{n-1} \in G^{(n-1)} (S)$ lying over $g_{n-1}$, where $R \to S$ is an \'etale $k$-algebra homomorphism. Then by formal smoothness of $q_0$, there is a simultaneous lift in $S$-points $\tilde{g}_n \in G^{(n)} (S)$ of both $g_n$ and $\tilde{g}_{n-1}$. Indeed, $g_n$ corresponds to a morphism $\Spec{(R[\varpi]/(\varpi^{n+1}))} \to G/H$ (and by precomposition, to a morphism $\Spec{(S[\varpi]/(\varpi^{n+1}))} \to G/H$), and $\tilde{g}_{n-1}$ corresponds to a morphism $\Spec{(S[\varpi]/(\varpi^n))} \to G$ such that $q_0 \circ \tilde{g}_{n-1}$ is equal to $j_n(g_n)$ as a morphism $\Spec{(S[\varpi]/(\varpi^n))} \to G/H$. Then by the infinitesimal lifting property of formally smooth morphisms, there is a lift $\tilde{g}_n$ in the diagram below.
    
    \[
     \begin{tikzcd}
      \Spec{(S[\varpi]/(\varpi^n))} \ar[r, "\tilde{g}_{n-1}"] \ar[d] &
      G \ar[d, "q_0"]
       \\
      \Spec{(S[\varpi]/(\varpi^{n+1}))}  \ar[ur, dashed, "\tilde{g}_n"] \ar[r, "g_n"] &
      G/H
     \end{tikzcd}
    \]
    
    Now suppose $g \in L^+(G/H)(R)$. For each $n$ there is a corresponding element $g_n \in (G/H)^{(n)}(R)$. In particular, there is an element $g_0 \in (G/H)(R)$ with a lift in $S$-points $\tilde{g}_0 \in G(S)$, for some \'etale $R \to S$. By above, for each $n$ there is also a lift $\tilde{g}_n \in (G/H)^{(n)}(S)$ of $g_n$. These $\tilde{g}_n$ form an inverse system, and thus correspond to an element $\tilde{g} \in L^+G(S)$ lifting $g$. Thus $q$ is surjective. 
    
        Now consider the kernel of $q$. Let $g \in L^+G(R)$ for a $k$-alebra $R$, and suppose $q(g) = e \in L^+(G/H)(R)$. Then $g \in \ker{(q_0)}(R\bbrackets{\varpi}) = H(R\bbrackets{\varpi}) = L^+H(R)$. Similarly, for $g \in L^+H(R) = H(R\bbrackets{\varpi})$, we have $q(g)$ corresponds to the identity in $(G/H)(R\bbrackets{\varpi}) = L^+(G/H)(R)$, and so $g \in \ker{q}(R)$.
   \end{proof}

   \begin{Lemma}
    \label{Lemma:quotient-L}
    Suppose $G$ is an \'etale group scheme over a field $k$. Then there are canonical isomorphisms $L^+G \cong LG \cong G$.
   \end{Lemma}
   
   \begin{proof}
    Let $R$ be a $k$-algebra, and let $k[G] = \Gamma(G, \Oo_G)$; note $k[G]$ is a direct product of fields, all finite and separable over $k$. I will show $L^+G(R) = LG(R) = G(R)$. The respective $R$-point sets are equal to $\HomA{\Alg{k}}{k[G]}{R\bbrackets{\varpi}}$, $\HomA{\Alg{k}}{k[G]}{R\laurents{\varpi}}$, and $\HomA{\Alg{k}}{k[G]}{R}$. It is sufficient to show that for all $f \in k[G]$ and all homomorphisms $\phi: k[G] \to R\laurents{\varpi}$, we have $\phi(f) \in R$. I first show that $\phi(f) \in R\bbrackets{\varpi}$, and then assume $\phi: k[G] \to R\bbrackets{\varpi}$ to show $\phi(f) \in R$. 
    
    We may assume $k[G]$ is a field: since $f = (f_1, \ldots, f_r)$ and $\phi = (\phi_1, \ldots, \phi_r)$, where $k[G]$ is a product of $r$ fields, it is sufficient to prove $\phi_i(f_i) \in R$ for all $i$. Both $f$ and $\phi(f)$ are invertible and satisfy a separable, irreducible polynomial over $k$. Let $\phi(f) = (a_n \varpi^n)_{n \in \ZZ}$, where $a_n \in R$ for all $n$ and $a_n = 0$ for all sufficiently small $n$, and let $n_0$ be the smallest integer such that $a_{n_0} \ne 0$. Invertibility of $\phi(f)$ implies that $a_{n_0}$ is invertible in $R$, and in particular not nilpotent. Then algebraicity of $\phi(f)$ over $k$ implies that $n_0 \ge 0$. Therefore $\phi(f) \in R\bbrackets{\varpi}$.
    
    Now we may assume $\phi: k[G] \to R\bbrackets{\varpi}$. Since $k[G]$ is a field, $\phi$ must be injective, even after composition. In particular, the composition $k[G] \to R\bbrackets{\varpi} \to R$ injective, where the second map sends $\varpi \mapsto 0$. Therefore $\phi(k[G]) \subset R$, and so $G(R\laurents{\varpi}) = G(R\bbrackets{\varpi}) = G(R)$.
   \end{proof}
   
   \begin{Prop}
    \label{Prop:no-circ}
    Let $G^{\sigma}$ be a possibly disconnected, split, reductive group over a field $k$, and let $G^{\sigma, \circ}$ be the identity component. Then the natural map of functors
    \[
     \begin{tikzcd}
      \GrG{G^{\sigma, \circ}} \ar[r, "\eta"] &
      \GrG{G^{\sigma}}
     \end{tikzcd}
    \]
    is an isomorphism of \'etale sheaves over $k$.
  \end{Prop}
   
  \begin{proof}
   Recall the affine Grassmannian is the \'etale sheafification of a presheaf $P\GrG{G^{\sigma}}: \Alg{k} \to \Sets$ defined by $R \mapsto LG^{\sigma}(R)/L^+G^{\sigma}(R)$. In order to prove that $\eta$ is an isomorphism of sheaves, it is sufficient to prove $\eta$ is both injective and surjective (as a map of sheaves).
   
   Injectivity of $\eta$ follows from injectivity of the the presheaf map $\eta^P: P\GrG{G^{\sigma, \circ}} \to P\GrG{G^{\sigma}}$. Let $R$ be an arbitrary $k$-algebra. Then the component map $\eta^P_R$ is injective. Indeed, an element $x \in P\GrG{G^{\sigma, \circ}}(R)$ can be written as a coset $x = gL^+G^{\sigma, \circ}(R)$, where $g \in LG^{\sigma, \circ}(R)$, and $\eta^P_R(x) = gL^+G^{\sigma}(R)$. This definition of $\eta^P_R$ makes sense because $LG^{\sigma, \circ} \subset LG^{\sigma}$ and is well-defined since $L^+G^{\sigma, \circ} \subset L^+G^{\sigma}$. The map $\eta^P_R$ is also injective. Indeed, if $g$ and $g'$ are two elements of $LG^{\sigma, \circ}(R)$ such that $gL^+G^{\sigma}(R) = g'L^+G^{\sigma}(R)$, let $h \in L^+G^{\sigma}(R)$ be any element such that $gh = g'$. Then in fact $h \in L^+G^{\sigma, \circ}(R)$, otherwise $gh \not\in LG^{\sigma, \circ}(R)$.
   
   In order to show that $\eta$ is surjective, I find it convenient to sheafify; the aim is to show that for all $k$-algebras $R$, and for all $x \in \GrG{G^{\sigma}}(R)$, there is some \'etale $k$-algebra morphism $R \to S$ such that $x$, viewed by restriction as a point in $\GrG{G^{\sigma}}(S)$, lifts to a point $\tilde{x} \in \GrG{G^{\sigma, \circ}}(S)$. To do so, consider the following diagram of \'etale sheaves, for which the rows are exact (as sheaves in pointed sets):
   \[
    \begin{tikzcd}
     1 \ar[r] &
     L^+G^{\sigma, \circ} \ar[r] \ar[d] &
     LG^{\sigma, \circ} \ar[r] \ar[d] &
     \GrG{G^{\sigma, \circ}} \ar[r] \ar[d, "\eta"] &
     1
      \\
     1 \ar[r] &
     L^+G^{\sigma} \ar[r] &
     LG^{\sigma} \ar[r] &
     \GrG{G^{\sigma}} \ar[r] &
     1
    \end{tikzcd}
   \]
   
   Let $x \in \GrG{G^{\sigma}}(R)$. By surjectivity of $LG^{\sigma} \to \GrG{G^{\sigma}}$, there is some lift $g \in LG^{\sigma}(S_1)$ of $x$, where $R \to S_1$ is \'etale. Then if we can find some $h \in L^+G^{\sigma}(S_2)$ such that $gh^{-1} \in LG^{\sigma, \circ}(S_2)$ (where $S_1 \to S_2$ is \'etale), then $\eta([gh^{-1}]) = x$. The reason we can find such $h$ (and $S_2$) is that the \'etale quotient functors $L^+G^{\sigma}/L^+G^{\sigma, \circ}$ and $LG^{\sigma}/LG^{\sigma, \circ}$ are isomorphic---in fact, they are isomorphic to the algebraic group $G^{\sigma}/G^{\sigma, \circ}$, so we can take $h \in G^{\sigma}(S_2)$. In particular, if $h' \in LG^{\sigma}(S_1)$ is any point with $g(h')^{-1} \in LG^{\sigma, \circ}(S_1)$, then $[h'] \in (LG^{\sigma}/LG^{\sigma, \circ})(S_1) = (G^{\sigma}/G^{\sigma, \circ})(S_1)$ lifts to $h \in G^{\sigma}(S_2)$ for some $S_2$.
   
   Use the previous two lemmas to see that the sheaves $L^+G^{\sigma}/L^+G^{\sigma, \circ}$ and $LG^{\sigma}/LG^{\sigma, \circ}$ really are isomorphic. The group $G^{\sigma, \circ}$ is reductive and a normal subgroup of the reductive group $G^{\sigma}$, and the quotient $G^{\sigma}/G^{\sigma, \circ}$ is \'etale. Thus by Lemma \ref{Lemma:quotient-L-pos}, $L^+G^{\sigma}/L^+G^{\sigma, \circ} \cong L^+(G^{\sigma}/G^{\sigma, \circ})$, and by Lemma \ref{Lemma:quotient-L}, $G^{\sigma}/G^{\sigma, \circ} \cong L^+(G^{\sigma}/G^{\sigma, \circ})$.
   
   It remains to be seen that $LG^{\sigma}/LG^{\sigma,\circ} \cong L(G^{\sigma}/G^{\sigma, \circ})$. Note that we have a natural map $LG^{\sigma} \to L(G^{\sigma}/G^{\sigma, \circ})$ by applying the loop group functor $L$ to the quotient map $q_0: G^{\sigma} \to G^{\sigma, \circ}$. By Lemma \ref{Lemma:quotient-L}, $L(G^{\sigma}/G^{\sigma, \circ}) \cong G^{\sigma}/G^{\sigma, \circ} \cong L^+(G^{\sigma}/G^{\sigma, \circ})$. Then the map $LG^{\sigma} \to L(G^{\sigma}/G^{\sigma, \circ})$ is surjective, since the map $L^+G^{\sigma} \to L^+(G^{\sigma}/G^{\sigma, \circ})$ factors through it, and is itself surjective by Lemma \ref{Lemma:quotient-L-pos}. And the kernel is $LG^{\sigma, \circ}$, for reasons essentially identical to those in the proof of Lemma \ref{Lemma:quotient-L-pos}: if the quotient map kills $g \in LG^{\sigma}(R)$, then $g$ corresponds to an element of $G^{\sigma}(R\laurents{\varpi})$ also killed by quotient, and thus $g \in G^{\sigma, \circ}(R\laurents{\varpi}) = LG^{\sigma, \circ}(R)$. And if $g \in LG^{\sigma, \circ}(R)$, then the quotient map kills $g$ when  viewed as an $R\laurents{\varpi}$-point of $G^{\sigma, \circ}$.

  \end{proof}
  
  It follows from the same reasoning that the natural map $G^{\sigma, \circ} \hookrightarrow G^{\sigma}$ induces an isomorphism of categories 
  \begin{equation}
   \label{eq:P-LG-GRG}
   P_{L^+G^{\sigma, \circ}} (\GrG{G^{\sigma, \circ}}) \overset{\sim}{\to}
   P_{L^+G^{\sigma}} (\GrG{G^{\sigma}}).
  \end{equation}
  Specifically, recall that the cocharacter lattices $X_*(T^{\sigma})$ and $X_*(T^{\sigma, \circ})$ are isomorophic. And for each cocharacter $\mu \in X_*(T^{\sigma})^+$, the map $\eta$ restricts to an isomorphism $\GrG{G^{\sigma}}^{\mu} \cong \GrG{G^{\sigma, \circ}}^{\mu}$. I do not directly use this fact, but I think it is worth acknowledging.
  
  I will prove equation (\ref{eq:main}) of Theorem \ref{THM:main} holds using $G^{\sigma}$. Then since the two categories in equation (\ref{eq:P-LG-GRG}) are isomorphic, they share a Tannakian dual group, which by Theorem \ref{THM:MV} \ref{THM:MV:v} has root datum dual to the connected reductive group $G^{\sigma, \circ}$, i.e. is isomorphic to the group $\widehat{G^{\sigma, \circ}}$.
   
   Much of the proof flows from an understanding of the relationship between the $\sigma$-action on semi-infinite cells, the $\sigma$-fixed sub-ind-scheme of a $\sigma$-invariant semi-infinite cell, and the corresponding semi-infinite cell of the $\sigma$-fixed point affine Grassmannian.
   
   \begin{Prop}
    \label{Prop:Orbits}
    \begin{enumerate}[label=\roman*.]
     \item
      \label{Prop:Orbits:i}
      There is a natural embedding of ind-schemes $\GrG{G^{\sigma}} \hookrightarrow \GrG{G}$, and $\GrG{G^{\sigma}}$ can be identified with $(\GrG{G})^{\sigma}$.
   
      \item
       \label{Prop:Orbits:ii}
       For $\mu \in X_*(T)^+$, $\sigma(\GrG{G}^{\mu}) = \GrG{G}^{\sigma(\mu)}$
       
      \item
       \label{Prop:Orbits:iii}
       For $\mu \in X_*(T)^{+, \sigma}$, we can identify $(\GrG{G}^{\mu})^{\sigma}=\GrG{G^{\sigma}}^{\mu}$.
       
      \item
       \label{Prop:Orbits:iv}
       For $\nu \in X_*(T)$ and $w \in W$, $\sigma(S_w^{\nu}) = S_{\sigma(w)}^{\sigma(\nu)}$
       
      \item
       \label{Prop:Orbits:v}
       For $\nu \in X_*(T)^{\sigma}$ and $w \in W^{\sigma}$, we can identify $(S_w^{\nu})^{\sigma} = (S_{\sigma})_w^{\nu}$, where $(S_{\sigma})_w^{\nu}$ is the semi-infinite cell $wL(U^{\sigma})w^{-1} \cdot \varpi^{\nu} x_0 \subset \GrG{G^{\sigma}}$.
     \end{enumerate}
   \end{Prop}

   \begin{proof}
    Statements \ref{Prop:Orbits:ii} and \ref{Prop:Orbits:iv} are immediate.
   
    Note that $\sigma$ acts on $LG$, preserving $L^+G$. Thus we have, a priori, an action of $\sigma$ on $\GrG{G}$ and an injective map of functors $\GrG{G^{\sigma}} \to (\GrG{G})^{\sigma}$.

    Statements \ref{Prop:Orbits:i} and \ref{Prop:Orbits:iii} follow from statement \ref{Prop:Orbits:v}, along with the observation that for a sub-ind-scheme $X \subset \GrG{G}$, $X^{\sigma} = X \cap (\GrG{G})^{\sigma}$.
    
    To see statement \ref{Prop:Orbits:v}, consider the action of $wLUw^{-1}$ on $S_w^{\nu}$: there is a subgroup, $J_G(w, \nu) \subset wLUw^{-1}$, with a simply transitive action on $S_w^{\nu}$. Define $J_G(w, \nu)$ as follows:
    \[
     J_G(w, \nu) :=
     wLUw^{-1} \cap \varpi^{\nu}L^{--}G\varpi^{-\nu}.
    \]
    By construction of $J_G(w, \nu)$, it is clear that for $\sigma$-invariant $w$ and $\nu$, $J_G(w, \nu)^{\sigma} = J_{G^{\sigma}}(w, \nu)$. Therefore the $\sigma$-fixed points of $S_w^{\nu}$ are exactly those in the orbit of the $\sigma$-fixed subgroup $J_{G^{\sigma}}(w, \nu)$.
    
    So let us see that the action of the subgroup $J_G(w, \nu)$ is simply transitive on $S_w^{\nu}$, implying Statement \ref{Prop:Orbits:v}. 
%
It is well known that $LU$ has a decomposition $L^{--}U \cdot L^+U$. Since $L^{--}G \supset L^{--}U$ acts freely on $\GrG{G}$ at the basepoint $x_0$ with stabilizer $L^+G \supset L^+U$, this decomposition implies that $J_G(e, 0) = L^{--}U$ acts simply transitively on $S_e^0$. Similarly, we have $J_G(w, 0)$ acting simply transitively on $S_w^0$ for all $w \in W$.
    
    For more general $J_G(w, \nu)$, consider the decomposition
    \[
     wLUw^{-1} =
     J_G(w, \nu) \cdot (wLUw^{-1} \cap \varpi^{\nu} L^+G \varpi^{-\nu}).
    \]
    It is immediate both that this is a decomposition of $wLUw^{-1}$ (from normality of $wUw^{-1}$ in $wBw^{-1}$), and also that $wLUw^{-1} \cap \varpi^{\nu} L^+G \varpi^{-\nu}$ is the stabilizer of $\varpi^{\nu}x_0$ in $wLUw^{-1}$. Therefore $J_G(w, \nu)$ acts simply transitively on $S_w^{\nu}$, as needed.

   \end{proof}
   
   Note also that closure relations hold as expected, simply by intersection. Specifically, given a $\sigma$-invariant dominant cocharacter $\mu$,
   \[
    \overline{\GrG{G^{\sigma}}^{\mu}} =
    \overline{\GrG{G}^{\mu}} \cap \GrG{G^{\sigma}} =
    (\coprod_{\substack{\lambda \in X_*(T)^+ \\ \lambda \le \mu}} \GrG{G}^{\lambda}) \cap \GrG{G^{\sigma}} =
    \coprod_{\substack{\lambda \in X_*(T)^+ \\ \lambda \le \mu}} (\GrG{G}^{\lambda} \cap \GrG{G^{\sigma}}) =
    \coprod_{\substack{\lambda \in X_*(T)^{+,\sigma} \\ \lambda \le \mu}} \GrG{G^{\sigma}}^{\lambda},
   \]
   and given $\sigma$-invariant $\nu \in X_*(T)$ and $w \in W^{\sigma}$,
   \[
    \overline{(S_{\sigma})_w^{\nu}} =
    \overline{S_w^{\nu}} \cap \GrG{G^{\sigma}} =
    (\coprod_{\substack{\eta \in X_*(T) \\ \eta \le_w \nu}} S_w^{\eta}) \cap \GrG{G^{\sigma}} =
    \coprod_{\substack{\eta \in X_*(T) \\ \eta \le_w \nu}} (S_w^{\eta} \cap \GrG{G^{\sigma}}) =
    \coprod_{\substack{\eta \in X_*(T)^{\sigma} \\ \eta \le_w \nu}} (S_{\sigma})_w^{\eta}.
   \]
   Either of these equalities implies that the cocharacter lattice $X_*(T^{\sigma}) \subset X_*(T)$ inherits the partial order $\le$. This can also be seen combinatorially; see Section \ref{S:Haines} for details.
   
   The primary varieties in consideration in this proof are Mirkovi\'c--Vilonen (MV) cycles and Anderson--Mirkovi\'c--Vilonen (AMV) cycles. MV cycles of coweight $(\lambda, \mu)$ are irreducible components of the intersection $S_{w_0}^{\lambda} \cap \overline{\GrG{G}^{\mu}}$, and according to the geometric Satake correspondence they index a basis for $V_{\mu}(\lambda)$. On the other hand, AMV cycles of coweight $(\lambda, \mu)$ are irreducible components of the variety $\overline{S_{w_0}^{\lambda} \cap S_e^{\mu}}$. Many authors refer to AMV cycles as simply ``MV cycles.'' The following proposition will make clear how closely related they are, and why AMV cycles may be considered a generalization of MV cycles. Note that, for the purposes of the geometric Satake equivalence, it is not important whether we are dealing with an equi-dimensional variety or its closure. Indeed, we can use top-dimensional cohomology with compact support, which in this case depends only on dimension and number of components. However, it is (formally) convenient to require AMV cycles to be projective when defining their moment polytopes.
   
   \begin{Prop}[\cite{An03} Proposition 3]
    \label{Prop:Polytope}
    If $A$ is an irreducible component of $\overline{S_{w_0}^{\lambda} \cap S_e^{\mu}}$ and $A \subset \overline{\GrG{G}^{\mu}}$, then $A$ is the closure of an MV cycle of coweight $(\lambda, \mu)$. If $A'$ is an MV cycle of coweight $(\lambda, \mu)$, then $\overline{A'}$ is an irreducible component of $\overline{S_{w_0}^{\lambda} \cap S_e^{\mu}}$. Thus the closures of MV cycles of coweight $(\lambda, \mu)$ are exactly the AMV cycles of coweight $(\lambda, \mu)$ contained in $\overline{\GrG{G}^{\mu}}$.
  \end{Prop}
  \begin{proof}
    This result follows from dimension estimates in Theorem \ref{THM:MV} \ref{THM:MV:i} and \ref{THM:MV:ii}.
    
    First suppose $A$ is an irreducible component of $\overline{S_{w_0}^{\lambda} \cap S_e^{\mu}}$ and that $A \subset \overline{\GrG{G}^{\mu}}$. Of course we have $A \subset \overline{S_{w_0}^{\lambda}} \cap \overline{\GrG{G}^{\mu}}$. Now the Iwasawa stratification implies
    \[
     \overline{S_{w_0}^{\lambda}} \cap \overline{\GrG{G}^{\mu}} =
     \coprod_{\substack{\nu \in X_*(T) \\ \nu \ge \lambda}} (S_{w_0}^{\nu} \cap \overline{\GrG{G}^{\mu}}) =
     (S_{w_0}^{\lambda} \cap \overline{\GrG{G}^{\mu}}) \cup X,
    \]
    where $\dim{X} < \dim{A}$. And so $A' := A \cap (S_{w_0}^{\lambda} \cap \overline{\GrG{G}^{\mu}})$ is dense in $A$. Since $A$ is irreducible, this implies $A'$ is an MV cycle.
    
    Now suppose $A = \overline{A'}$ where $A'$ is an MV cycle of coweight $(\lambda, \mu)$. It is sufficient to see that $A' \subset \overline{S_{w_0}^{\lambda} \cap S_e^{\mu}}$. Note that by Theorem \ref{THM:MV:3.2}(a), 
    \[
     \dimp{S_e^{\mu} \cap \overline{\GrG{G}^{\mu}}} =
     \dimp{\overline{\GrG{G}^{\mu}}} =
     2\langle{\rho, \mu}\rangle,
    \]
    implying that $S_e^{\mu} \cap \overline{\GrG{G}^{\mu}}$ is dense in $\overline{\GrG{G}^{\mu}}$, and in particular, $\overline{\GrG{G}^{\mu}} \subset \overline{S_e^{\mu}}$. So $A' \subset S_{w_0}^{\lambda} \cap \overline{S_e^{\mu}}$. Again using the Iwasawa stratification, we have
    \[
     S_{w_0}^{\lambda} \cap \overline{S_e^{\mu}} =
     \coprod_{\substack{\nu \in X_*(T) \\ \nu \le \mu}} (S_{w_0}^{\lambda} \cap S_e^{\nu}) =
     (S_{w_0}^{\lambda} \cap S_e^{\mu}) \cup Y,
    \]
    where again $\dim{Y} < \dim{A'}$. And so $A' \cap (S_{w_0}^{\lambda} \cap S_e^{\mu})$ is dense in $A'$, and $A' \subset \overline{S_{w_0}^{\lambda} \cap S_e^{\mu}}$.
  \end{proof}
   
   Working with AMV cycles rather than MV cycles is convenient. The primary reason is that they are defined as components of a pair of semi-infinite cells, rather than components of a semi-infinite cell and a Schubert variety. One useful consequence is that $X_*(T)$ acts on the set of AMV cycles by translation, whereas a translation of an MV cycle is no longer necessarily an MV cycle.
   
   Given an AMV cycle $A$ and a cocharacter $\nu \in X_*(T)$, let $\nu \cdot A = \varpi^{\nu}A$. We have $\varpi^{\nu} S^{\lambda}_w = S^{\lambda + \nu}_w$ by normality of $wUw^{-1}$ in $wBw^{-1}$, so if $A$ has coweight $(\lambda, \mu)$, then $\nu \cdot A$ is an AMV cycle and has coweight $(\lambda + \nu, \mu + \nu)$. Given an AMV cycle $A$ of coweight $(\lambda, \mu)$, the $X_*(T)$-orbit of $A$ has one AMV cycle of coweight $(\lambda - \mu, 0)$. This AMV cycle is called the stable AMV cycle representing $A$, and denoted $A_0$. For many purposes, I will work with stable AMV cycles only, equivalent to assuming the second coweight is $0$. In these cases I will use the subscript ${}_0$. This is especially convenient for the consideration of non-simply-connected groups $G$, since stable AMV cycles are contained in the neutral component of the affine Grassmannian.
   
   The following theorem of Anderson is useful for determining which AMV cycles are MV cycles.
   
   \begin{THM}[\cite{An03} Theorem 1 (1)]
    \label{THM:Polytope}
    Let $G$ be a semisimple group over $\CC$. There exists a family of polytopes $\mathcal{MV} = (P_A)_{A \in \BB}$ in $X_*(T)_{\RR}$ with parameter set $\BB$ graded by $\Lambda^-$ (i.e. $\BB = \bigcup_{\nu \in \Lambda^-} \BB_{\nu}$) such that weight multiplicities may be calculated according to the following rule: If $V_{\mu}$ is an irreducible representation of $\widehat{G}$ with highest weight $\mu$, then the multiplicity of the weight $\lambda$ in $V_{\mu}$ equals the number of $A \in \BB_{\lambda-\mu}$ for which $P_A + \mu \subset \mathrm{Conv}(W \cdot \mu)$.
  \end{THM}
  
  Above, $X_*(T)_{\RR} = X_*(T) \otimes_{\ZZ} \RR$, $\BB_{\nu}$ is the set of irreducible components of $\overline{X(\nu, 0)}$, and $\Lambda^-$ is the set of negative coweights in $X_*(T^{sc})$, i.e. the negative coroot semilattice of $G$. And $P_A$ is the moment polytope of the MV cycle $A$, defined as follows:
  
  \begin{DEF}[Moment polytope]
   Suppose $X$ is an irreducible, projective, $T$-invariant subvariety $X \subset \GrG{G}$. Then define the moment polytope of $X$ as
   \[
    P_X :=
    \Conv{\nu \in X_*(T) \mid \varpi^{\nu} x_0 \in X}.
   \]
   
   This definition is inspired by the image of the moment map $\Phi: \GrG{G} \to X_*(T)_{\RR}$ of the action of $T$ on $\GrG{G}$. However, for our purposes, there is no need to define $\Phi$, only the image of $T$-invariant subvarieties.
  \end{DEF}
  
  Note that Schubert varieties, semi-infinite cells, and AMV cycles are all $T$-invariant. It is helpful to note some properties of the moment map and polytopes it produces:
  \begin{Prop}[\cite{An03} Proposition 4 and proof]
   \label{Prop:moment}
   \begin{enumerate}[label=\roman*.]
    \item
     \label{Prop:moment-i}
     The $T$-fixed points of $\GrG{G}$ are the $\varpi^{\nu}x_0$. Those in $\GrG{G}^{\mu}$ are the $\varpi^{W \cdot \mu}x_0$. Those in $\overline{\GrG{G}^{\mu}}$ are the $\varpi^{\nu}x_0$ where 
     \[
      \nu \in \Conv{W \cdot \mu} \cap (\mu + X_*(T^{sc})).
     \]
     The one in $S_w^{\nu}$ is $\varpi^{\nu} x_0$. Those in $\overline{S_w^{\nu}}$ are the $\varpi^{\eta} x_0$ where $\eta \le_w \mu$.
     
    \item
     \label{Prop:moment-ii}
     If $X$ is a one-dimensional $T$-orbit, then $P_{\overline{X}}$ is a line segment in a coroot direction joining two coweights in a common coset modulo $X_*(T^{sc})$.
     
    \item
     \label{Prop:moment-iii}
     If $X$ is any projective, irreducible, $T$-invariant variety, then $P_X$ is the convex hull of the images of its $T$-fixed points.
     
   \end{enumerate}
  \end{Prop}
  \begin{proof}
   Statement \ref{Prop:moment-i} is well known. Statement \ref{Prop:moment-iii} is an immediate consequence of the definition of the moment polytope used here.
   
   Suppose $X \subset \GrG{G}$ is a one-dimensional $T$-orbit. Let $x$ be a complex point $x \in X(\CC)$. By the Iwasawa stratification, for all $w \in W$ there is a unique $\nu_w$ for which $x \in S_w^{\nu_w}$. For each $w$, the set of $T$-fixed points in $\overline{S_w^{\nu_w}}$ is $\{\varpi^{\eta} x_0 \mid{}$ $\eta \le_w \nu_w\}$. So the set of $T$-fixed points in the intersection of $S_w^{\nu_w}$ for all $w \in W$ is contained in $\Conv{\nu_w \mid w \in W}$. In particular, $P_{\overline{X}} \subset \Conv{\nu_w \mid w \in W}$.
   
   In fact, $P_{\overline{X}} = \Conv{\nu_w \mid w \in W}$. Recall the geometric description of the semi-infinite cells (\ref{eq:semi-infinite}): $x \in S_w^{\nu} (\CC)$ if and only if
   \[
    \lim_{t \to 0} w(\rho^{\vee}) (t) \cdot x = \varpi^{\nu} x_0.
   \]
   The $\varpi^{\nu_w} x_0$ for $w \in W$ are therefore limit points for the torus action on $X$, and contained in $\overline{X}$.
   
   Since $X$ is a one-dimensional quotient of $T$, it must be isomorphic to $\GG_m$, and so $X(\CC)$ has at most two distinct limit points in $\GrG{G}(\CC)$. But if there is only one, call it $\nu$, then $X(\CC) \subset (S_{w_0}^{\nu} \cap S_e^{\nu}) (\CC)$, which consists of the single point $\varpi^{\nu} x_0$, violating the assumption that $X$ is one-dimensional. So let the two distinct cocharacters in the set $\{\nu_w \mid w \in W\}$ be denoted $\nu$ and $\eta$. Suppose $\nu = \nu_e$, so that $\nu \ge \eta$. Then also we have $\eta = \nu_{w_0}$.
   
   Find some $w \in W$ and simple reflection $s_i \in W$ such that $\nu \ge_w \eta$ but $\nu \le_{ws_i} \eta$, and note that then $\ell(w) < \ell(w s_i)$. Such $w$ and $s_i$ can be found by choosing a reduced word $\mathbf{i} = (i_1, i_2, \ldots, i_{\ell(w_0)})$ for $w_0$ and comparing $\nu$ and $\eta$ under the order $\le_{s_{i_1} \cdots s_{i_k}}$ for each $0 \le k \le \ell(w_0)$. Then $w^{-1}(\nu) \le w^{-1}(\eta)$, so $w^{-1}(\nu -\eta)$ is a sum of positive coroots; and $s_i w^{-1}(\nu) \ge s_i w^{-1}(\eta)$, so $s_i w^{-1}(\nu - \eta)$ is a sum of negative coroots. This is only possible when $w^{-1}(\nu - \eta) = n\alpha^{\vee}_i$ for some integer $n > 0$. And so $\nu - \eta$ is a multiple of the positive coroot $w(\alpha_i)$.
  \end{proof}
  
  I will sketch Anderson's proof of Theorem \ref{THM:Polytope} below, to see that it does not depend on the assumption $G$ is semisimple.

   Although $T$-invariance is the primary consideration for the moment map, it is convenient to also consider the ``dilation'' action of $\GG_m$ on $\GrG{G}$ defined by $\varpi \mapsto c \varpi \in R\laurents{\varpi}$ for $c \in \GG_m(R)$, since we can usefully describe a fixed point in the closure of a $(\GG_m \times T)$-orbit. Note also that Schubert varieties, semi-infinite cells, and AMV cycles are dilation-invariant as well as $T$-invariant. The following statement about the fixed points of torus actions is well-known. The proof in the current case is taken from Anderson.
   
   \begin{Lemma}[\cite{An03} Lemma 6]
    \label{Lemma:fp}
    Every $(\GG_m \times T)$-orbit $X \subset \GrG{G}$ has a $T$-fixed point $\varpi^{\eta} x_0 \in \overline{X}$, such that $X \subset \GrG{G}^{\eta}$.
   \end{Lemma}
   
   Note that in the statement above, $\eta$ is not required to be dominant.
   
  \begin{proof}
   Suppose $X$ is a $(\GG_m \times T)$-orbit in $\GrG{G}$. By the Cartan stratification, $X$ must be contained in some $\GrG{G}^{\mu}$. There is a point $x \in \overline{X}(\CC) \cap (G(\CC) \cdot \varpi^{\mu} x_0)$, found as a limit of the dilation action. Then there is some fixed point $\varpi^{\eta} x_0$ in the closure of the $T(\CC)$-orbit of $x$. Since $T \subset G$, we still have $\varpi^{\eta} x_0 \in G(\CC) \cdot \varpi^{\mu} x_0$; in particular, $\eta$ is in the Weyl orbit $W \cdot \mu$. And since $X$ is both $\GG_m$- and $T$-invariant, both $x$ and $\varpi^{\eta} x_0$ are contained in $\overline{X}(\CC)$.
  \end{proof}
  
  Now we are ready to prove Theorem \ref{THM:Polytope}.
   
  \begin{proof}[Proof of Theorem \ref{THM:Polytope}.]
   First suppose $A$ is an MV cycle of coweight $(\lambda, \mu)$ dominant cocharacter $\mu$ and some $\lambda \in Wt(\mu)$. Then $A \subset \overline{\GrG{G}^{\mu}}$, and so $P_A \subset P_{\overline{\GrG{G}^{\mu}}} = \Conv{W \cdot \mu}$.
   
   Now suppose $A$ is an AMV cycle of coweight $(\lambda, \mu)$, and suppose $P_A \subset \Conv{W \cdot \mu}$. Then for every vertex $\nu$ of $P_A$ and each $w \in W$, we have $\nu \le_{w} w(\mu)$.  Since $A$ is $(\GG_m \times T)$-invariant, it is a union of $(\GG_m \times T)$-orbits. So every complex point $a \in A(\CC)$ is contained in such an orbit $a \in X(\CC) \subset A(\CC)$, and by Lemma \ref{Lemma:fp} there is a cocharacter $\eta$ such that $\varpi^{\eta} x_0 \in \overline{X}(\CC) \subset A(\CC)$ and $X \subset \GrG{G}^{\eta}$. In particular, $a \in \GrG{G}^{\eta}(\CC)$. By the assumption that $P_A \subset \Conv{W \cdot \mu}$, the difference $\mu - w(\eta)$ must be a nonnegative real combination of simple roots for each $w \in W$. And since $\varpi^{\eta} x_0 \in A(\CC)$, and $A$ is irreducible, it follows that all $w(\eta)$ are in the same coset modulo $X_*(T^{\sc})$ as the vertices $\nu$ of $P_A$---so in fact $\eta \in Wt(\mu)$. Therefore $a \in \overline{\GrG{G}^{\mu}}(\CC)$.
  \end{proof}
   
  The following observation summarizes the convenience of moment polytopes in studying the geometry of the affine Grassmannian via semi-infinite cells. It is a direct consequence of Anderson's work.
  
  \begin{Lemma}
   \label{Lemma:Polytope}
    Let $X$ be a $T$-invariant, projective, irreducible subvariety of $\GrG{G}$, and let $\nu \in X_*(T)$. Then $X \cap S_w^{\nu}$ is dense in $X$ if and only if $\nu$ is $\le_w$-maximal among vertices of the moment polytope $P_X$.
   \end{Lemma}
   
   \begin{proof}
    First suppose $X \cap S_w^{\nu}$ is dense in $X$. Then $X = X \cap \overline{S_w^{\nu}}$. So if $\eta \not{\le_w} \nu$, then $X \cap S_w^{\eta} = \varnothing$. In particular, $\varpi^{\eta}x_0 \not \in X$. And so $\eta \not \in P_X$, unless $\eta \not \in \nu + X_*(T^{sc})$, in which case $\eta$ cannot be a vertex of $P_X$.
    
    Now suppose $\nu$ is $\le_w$-maximal among vertices of $P_X$. By the Iwasawa stratification, $X \subset \overline{S_w^{\nu}}$. In particular,
    \[
     X =
     \coprod_{\eta \le_w \nu} (X \cap S_w^{\eta})
    \]
    and so
    \[
     X =
     \overline{\coprod_{\eta \le_w \nu} (X \cap S_w^{\eta})} =
     \bigcup_{\eta \le_w \nu} (\overline{X \cap S_w^{\eta}}),
    \]
    with the last equality holding since there are only finitely many semi-infinite cells intersecting $X$. By irreducibility and completeness of $X$, there is thus some $\eta \le_w \nu$ such that $X = \overline{X \cap S_w^{\eta}}$. But the only $\eta \le_w \nu$ with $\varpi^{\nu}x_0 \in S_w^{\eta}$ is $\eta = \nu$. Therefore $X = \overline{X \cap S_w^{\nu}}$.
   \end{proof}

  These results suggest the construction of GGMS strata: small, $T$-invariant subvarieties stratifying $\GrG{G}$, whose closures are sometimes AMV cycles. A GGMS stratum, named for Gelfand, Goresky, MacPherson, and Serganova, is an intersection of a sequence of semi-infinite cells indexed by $W$. A sequence $\nu_{\bullet} = (\nu_w)_{w \in W}$ of cocharacters such that $\nu_w \ge_w \nu_{w'}$ for all $w, w' \in W$ is known as a GGMS datum, and specifies the GGMS stratum
  \[
   A(\nu_{\bullet}) :=
   \bigcap_{w \in W} S_w^{\nu_w}.
  \]
  GGMS data are in bijection with pseudo-Weyl polytopes, or convex polytopes in $X_*(T)_{\RR}$ whose edges are in root directions, and whose vertices are cocharacters in a common coset modulo $X_*(T^{sc})$. Combining the information in Proposition \ref{Prop:moment}, we see that if $A$ is an AMV cycle, then $P_A$ is a pseudo-Weyl polytope, and the vertices of $P_A$ form a GGMS datum $\nu^A_{\bullet}$. By Lemma \ref{Lemma:Polytope}, the GGMS stratum $GGMS(A) := A(\nu^A_{\bullet})$ is dense in $A$, and is the minimal intersection of semi-infinite cells with this property.
   
   As a consequence of their construction as the intersections of semi-infinite cells, the $\sigma$-action on the set of GGMS strata adheres to the following dichotomy:
   
   \begin{Lemma}
    \label{Lemma:GGMS-dichotomy}
    \begin{enumerate}[label=\roman*.]
     \item 
      \label{Lemma:GGMS-dichotomy-i}
      A GGMS stratum $A(\nu_{\bullet})$ is $\sigma$-invariant if and only if the sequence of cocharacters $\nu_{\bullet}$ is $\sigma$-invariant, meaning $\sigma(\nu_w) = \nu_{\sigma(w)}$ for all $w \in W$. 
      
     \item 
      \label{Lemma:GGMS-dichotomy-ii}
      If $A(\nu_{\bullet})$ is not $\sigma$-invariant, then $A(\nu_{\bullet})^{\sigma} = \varnothing$.
    \end{enumerate}
   \end{Lemma}
   
   \begin{proof}
    First, note that $\GrG{G}$ is stratified by GGMS strata. Intersecting the Iwasawa stratifications of $\GrG{G}$ for all $w \in W$, we have
    \[
     \GrG{G} =
     \coprod_{\nu_{\bullet}} A(\nu_{\bullet}).
    \]
    where the disjoint union runs over all GGMS data $\nu_{\bullet}$.
    
    Then let $\nu_{\bullet} = (\nu_w)_{w \in W}$. We have
    \[
     \sigma(A(\nu_{\bullet})) =
     \sigma\left(\bigcap_{w \in W} S_w^{\nu_w}\right) =
     \bigcap_{w \in W} S_{\sigma(w)}^{\sigma(\nu_w)} =
     A(\sigma(\nu_{\bullet})),
    \]
    where $\sigma(\nu_{\bullet}) = (\sigma(\nu_{\sigma^{-1}(w)}))_{w \in W}$. Then $\sigma$ permutes GGMS strata, implying \ref{Lemma:GGMS-dichotomy-ii} And \ref{Lemma:GGMS-dichotomy-i} follows from observing that $\sigma(\nu_{\bullet}) = \nu_{\bullet}$ if and only if $\nu_w = \sigma(\nu_{\bullet})_w$ for each $w \in W$.
   \end{proof}
   
  \section{Indexing using $\mathbf{i}$-Lusztig data}
   \label{S:Lusztig}
   Unfortunately, not all GGMS strata are dense in some AMV cycle. In order to work with an indexing set where we can be sure the result is an AMV cycle, we use $\mathbf{i}$-Lusztig data. 
   
   Let $A$ be an AMV cycle with GGMS datum $\nu_{\bullet}$. If we walk along the edges of the moment polytope $P_A$ from $\nu_e$ to $\nu_{w_0}$ such that the vertices are indexed by Weyl elements of strictly increasing lengths, by recording the length of each edge we produce a sequence of $\ell(w_0)$ nonnegative integers $n_{\bullet} \in \ZZ_{\ge 0}^{\ell(w_0)}$ called the $\mathbf{i}$-Lusztig datum of $A$. This $\mathbf{i}$-Lusztig datum uniquely determines the stable AMV cycle $A_0$, where $\mathbf{i}$ is the corresponding reduced word for $w_0$. The $\mathbf{i}$-Lusztig strata are in a useful bijection with stable AMV cycles:
   
   \begin{THM}[\cite{Ka10} Theorem 4.2]
    \label{THM:Lusztig-index}
    Given fixed reduced word $\mathbf{i}$ for $w_0$, the set of $\mathbf{i}$-Lusztig data are in bijective correspondence with stable AMV cycles. 
    \[
     \ZZ_{\ge 0}^{\ell(w_0)} \leftrightarrow 
     \{\textnormal{stable AMV cycles}\}
    \]
   \end{THM}
   
   Note that since Theorem \ref{THM:Lusztig-index} concerns stable AMV cycles, there is no question that it applies to connected, reductive groups, and not just to simply connected, semisimple groups.
   
  \begin{proof}
   From the construction of the $\mathbf{i}$-Lusztig datum of an AMV cycle, it is sufficient to find an inverse mapping of $\mathbf{i}$-Lusztig data to stable AMV cycles. This can be done explicitly.
   
   Note $\Conv{W \cdot \rho^{\vee}}$ is a pseudo-Weyl polytope, and has $\mathbf{i}$-Lusztig datum $(1, \ldots, 1)$ for every reduced word $\mathbf{i}$ for $w_0$; indeed, for any two neighboring vertices $w(\rho^{\vee})$ and $ws_i(\rho^{\vee})$, the difference is a single coroot $w(\rho^{\vee} - s_i(\rho^{\vee})) = \pm w(\alpha^{\vee}_i)$. This polytope is sometimes known as the permutahedron. 
   
   One way to construct a subvariety of $\GrG{G}$ whose GGMS datum has a given $\mathbf{i}$-Lusztig datum $n_{\bullet}$ is to intersect semi-infinite cells correponding to the cocharacters encountered in the path from $\nu_e$ to $\nu_{w_0}$ corresponding to $\mathbf{i}$, the directions of which are determined by $\mathbf{i}$, and which can all be seen as the directions of edges of the permutahedron. It will turn out that this construction produces a collection of irreducible varieties in bijection with the $\mathbf{i}$-Lusztig data. 
  
    Fix $\mathbf{i} = (i_1, \ldots, i_{\ell(w_0)})$, a reduced word for $w_0$. Let $(w^{\mathbf{i}}_k)_{0 \le k \le \ell(w_0)}$ be the sequence of Weyl group elements corresponding to the the initial (as in leftmost, assuming $W$ acts from the left) $k$ letters of $\mathbf{i}$: $w^{\mathbf{i}}_0 = e$, $w^{\mathbf{i}}_k = w^{\mathbf{i}}_{k-1}s_{i_k}$, and $w^{\mathbf{i}}_{\ell(w_0)} = w_0$. 
    
     For $1 \le k \le \ell(w_0)$, let $\beta^{\mathbf{i},\vee}_k$ be the difference $w^{\mathbf{i}}_k(\rho^{\vee}) - w^{\mathbf{i}}_{k-1}(\rho^{\vee})$, which is the direction of an edge of the permutahedron. Then $\beta^{\mathbf{i}, \vee}_k$ is a negative coroot; specifically, $\beta^{\mathbf{i}, \vee}_k = -w^{\mathbf{i}}_{k-1}(\alpha_{i_k}^{\vee})$:
    \begin{align*}
     w^{\mathbf{i}}_k (\rho^{\vee}) - w^{\mathbf{i}}_{k-1} (\rho^{\vee}) &
     = w^{\mathbf{i}}_{k-1} (s_{i_k}(\rho^{\vee}) - \rho^{\vee})
      \\
     &
     = w^{\mathbf{i}}_{k-1} (s_{i_k}(\frac{1}{2} \sum_{\alpha^{\vee} \in \Phi^{\vee, +}} \alpha^{\vee}) - \frac{1}{2} \sum_{\alpha^{\vee} \in \Phi^{\vee, +}} \alpha^{\vee})
      \\
     &
     = w^{\mathbf{i}}_{k-1} (-\alpha_{i_k}^{\vee}),
    \end{align*}
    since $s_{i_k}$ permutes all positive coroots except for $\alpha^{\vee}_{i_k}$, which it transposes with $-\alpha^{\vee}_{i_k}$. Let $\nu_0 = 0$, and for each $k \ge 1$ let $\nu_k = \nu_{k-1} + \beta^{\mathbf{i}, \vee}_k$. Each $\nu_k$ is $\le_{w^{\mathbf{i}}_k}$-maximal in the sequence $(\nu_k)$, since $w^{\mathbf{i}}_k (\rho^{\vee})$ is $\le_{w^{\mathbf{i}}_k}$-maximal in $W \cdot \rho^{\vee}$. Then let
    \[
     A_0^{\mathbf{i}} (n_{\bullet}) := 
     \bigcap_{0 \le k \le \ell(w_0)} S_{w^{\mathbf{i}}_k}^{\nu_k}.
    \]
    Then $\overline{A_0^{\mathbf{i}}(n_{\bullet})}$ is a projective subvariety of $\overline{X(\nu_{\ell(w_0)}, 0)} = \overline{S_{w_0}^{\nu_{\ell{w_0}}} \cap S_e^0}$. Note also that the map
    \[
     \{n_{\bullet} \mid \mathbf{i}\text{-Lusztig data}\} \to
     \{A_0 \mid \text{stable AMV cycles}\}
    \]
    given by $n_{\bullet} \mapsto \overline{A_0^{\mathbf{i}}(n_{\bullet})}$ is injective.
   
    Suppose that $A_0^{\mathbf{i}}(n_{\bullet})$ is irreducible. Then it is not hard to see that $\overline{A_0^{\mathbf{i}}(n_{\bullet})}$ is a stable AMV cycle, proving the theorem. 
    
    Indeed, if $A_0$ is a stable MV cycle with GGMS datum $\nu_{\bullet}$ and $\mathbf{i}$-Lusztig datum $n_{\bullet}$, we have containment
    \[
     A_0(\nu_{\bullet}) \subset 
     A_0^{\mathbf{i}} (n_{\bullet}) \subset
     \overline{X(\nu_{w_0}, 0)}, \qquad
     A_0(\nu_{\bullet}) \subset 
     A_0 \subset
     \overline{X(\nu_{w_0}, 0)}.
    \]
    Then by Lemma \ref{Lemma:Polytope}, $A_0^{\mathbf{i}} (n_{\bullet}) \cap A_0$ is dense in $A_0$. Since $A_0$ is projective, the result follows from irreducibility of $A_0^{\mathbf{i}} (n_{\bullet})$.
    
    And all stable AMV cycles are constructed in this way: the number of stable AMV cycles of a given coweight $(\nu, 0)$ is the Kostant partition function of $\nu$, defined as the number of ways $\nu$ can be written as a combination of negative coroots. The reason for this is that the Kostant partition function is a sharp upper bound on the dimension of the weight space $V_{\mu}(\mu + \nu)$, and by Theorem \ref{THM:Polytope} every stable AMV cycle $A_0$ represents an MV cycle $ n\rho^{\vee} \cdot A_0$ for sufficiently large $n$. And since every negative coroot is the direction of exactly one edge in the paths defined by $\mathbf{i}$, the Kostant partition function of $\nu$ is exactly the number of distinct $\mathbf{i}$-Lusztig data of coweight $(\nu, 0)$.
    
    Irreducibility is somewhat difficult. The heart of the proof of irreducibility is the construction of surjective maps
    \[
     \mathbf{y}_{\mathbf{i}} : B(n_{\bullet}) \to A_0^{\mathbf{i}}(n_{\bullet}),
    \]
    where $B(n_{\bullet})$ is the subfunctor of $L\GG_a^{\ell(w_0)}$ defined on $R$-points as
    \[
     B(n_{\bullet})(R) =
     \{(x_1, \ldots, x_{\ell(w_0)}) \in R\laurents{\varpi}^{\ell(w_0)} \mid \valp{\varpi}{x_k} = n_k \text{ for } 1 \le k \le \ell(w_0)\}.
    \]
    
    Clearly $B(n_{\bullet})$ is an irreducible ind-scheme, so the existence of a surjective map of sheaves $\mathbf{y}_{\mathbf{i}}$ implies that $A_0^{\mathbf{i}}(n_{\bullet})$ is irreducible as well. For details and construction of $\mathbf{y}_{\mathbf{i}}$, see \cite{Ka10} Theorem 4.5. This construction is explicit, but somewhat complicated. Note that it does not depend in any way on simply-connectedness or semisimplicity of $G$.
   \end{proof}
   
   It is also the case that, fixing a reduced word $\mathbf{i}$, the variety $X(\nu, 0)$ is stratified by $\mathbf{i}$-Lusztig strata:
   \begin{equation}
    \label{eq:Lusztig-strat}
    X(\nu, 0) = 
    S_{w_0}^{\nu} \cap S_e^0 =
    \coprod_{\substack{n_{\bullet} \in \ZZ_{\ge 0}^{\ell(w_0)} \\ \nu_{\ell(w_0)} = \nu}} A_0^{\mathbf{i}} (n_{\bullet}).
   \end{equation}
   
   \begin{DEF}[$\sigma$-compatibility]
    A reduced word $\mathbf{i}$ for the longest element $w_0 \in W$ is said to be $\sigma$-compatible if there is a (uniquely determined) reduced word $\mathbf{i}_{\sigma}$ for the longest element $w_{0, \sigma} \in W^{\sigma}$ such that $\mathbf{i}$ is an expansion of $\mathbf{i}_{\sigma}$.
    
    The group $W^{\sigma}$ is a Coxeter group whose simple reflections are the longest elements in the subgroup generated by the simple reflections in a single $\sigma$-orbit (see Proposition \ref{Prop:Weyl}). If $\mathbf{i}_{\sigma}$ is a sequence of orbits $(\eta_1, \ldots, \eta_{\ell(w_{0, \sigma})})$, we say $\mathbf{i}$ is an expansion of $\mathbf{i}_{\sigma}$ if it can be partitioned into consecutive subsequences of letters corresponding to the orbits of $\mathbf{i}_{\sigma}$. If this is the case, then each such consecutive subsequence will be the longest word in the Coxeter subgroup generated by the simple reflections in the corresponding orbit.
   \end{DEF}
   
   \begin{Eg}
    \label{Eg:compatible}
    To illustrate the concept of $\sigma$-compatibility, consider the involution on the standard pinning of $SL_5$, with fixed-point subgroup $(SL_5)^{\sigma} \cong SO(5)$. On the set of simple roots, $\sigma$ acts by $(14)(23)$, with two orbits: $\eta_1 = \{1, 4\}$ and $\eta_2 = \{2, 3\}$. Then 
    \begin{align*}
     W = \langle s_1, s_2, s_3, s_4 & 
     {} \mid m_{13} = m_{14} = m_{24} = 2, \ m_{12} = m_{23} = m_{34} = 3 \rangle \\
     W^{\sigma} = \langle s_{\eta_1}, s_{\eta_2} &
     {} \mid m_{\eta_1 \eta_2} = 4 \rangle,
    \end{align*}
    so one (of two) possible reduced word for $w_{0, \sigma}$ is $\mathbf{i}_{\sigma} = (\eta_1, \eta_2, \eta_1, \eta_2)$. Since $m_{14} = 2$, the longest words on the letters in $\eta_1$ are $(1,4)$ and $(4,1)$. Since $m_{23} = 3$, the longest words on the letters in $\eta_2$ are $(2, 3, 2)$ and $(3, 2, 3)$. There are $16$ reduced words for $w_0$ expanding $\mathbf{i}_{\sigma}$, including
    \[
     \mathbf{i} = (1, 4, 2, 3, 2, 1, 4, 2, 3, 2).
    \]
   \end{Eg}
   
   \begin{Prop}
    \label{Prop:sigma-Lusztig}
    If $A_0^{\mathbf{i}} (n_{\bullet})$ is an $\mathbf{i}$-Lusztig stratum, then $\sigma(A_0^{\mathbf{i}} (n_{\bullet})) = A_0^{\sigma(\mathbf{i})} (n_{\bullet})$.
   \end{Prop}
    
   \begin{proof}
    Let $\mathbf{i}$ be a reduced word for $w_0$, and $n_{\bullet}$ an $\mathbf{i}$-Lusztig datum. Consider
    \[
     \sigma(A_0^{\mathbf{i}} (n_{\bullet})) =
     \sigma \left( \bigcap_{0 \le k \le \ell(w_0)} S_{w^{\mathbf{i}}_k}^{\nu_k} \right) =
     \bigcap_{0 \le k \le \ell(w_0)} S_{\sigma(w^{\mathbf{i}}_k)}^{\sigma(\nu_k)}.
    \]
    Now $\sigma(w^{\mathbf{i}}_k) = w^{\sigma(\mathbf{i})}_k$, and each $\sigma(\nu_k) = \sigma(\nu_{k-1}) + n_k \sigma(\beta^{\mathbf{i}, \vee}_k)$. We have 
    \[
     \sigma(\beta^{\mathbf{i}, \vee}_k) = 
     \sigma(w^{\mathbf{i}}_k(\rho^{\vee}) - w^{\mathbf{i}}_{k-1}(\rho^{\vee})) =
     w^{\sigma(\mathbf{i})}_k(\rho^{\vee}) - w^{\sigma(\mathbf{i})}_{k-1}(\rho^{\vee}) =
     \beta^{\sigma(\mathbf{i}), \vee}_k.
    \]
    By induction the $\sigma(\mathbf{i})$-Lusztig datum for $\sigma(A_0^{\mathbf{i}}(n_{\bullet}))$ is thus also $n_{\bullet}$.
   \end{proof}
    
    Suppose $\mathbf{i}$ and $\mathbf{i}'$ are distinct reduced words for $w_0$ that both expand a common reduced word $\mathbf{i}_{\sigma}$ for $w_{0, \sigma}$. Let $n_{\bullet}$ be an $\mathbf{i}$-Lusztig datum. Then there is an explicit procedure for producing the $\mathbf{i}'$-Lusztig datum $n'_{\bullet}$ for $A_0^{\mathbf{i}} (n_{\bullet})$.
    
    \begin{Lemma}[\cite{Ka10} Proposition 5.2]
     \label{Lemma:Lusztig-transforms}
     Let $\mathbf{i}$ and $\mathbf{i}'$ be two reduced words for $w_0$ related by a braid move corresponding to a pair of simple roots either disconnected or connected by an edge of weight $1$. Specifically, let $i_k$ and $i_{k+1}$ be the indices of a pair of simple roots, and let $m_{i_k, i_{k+1}} \le 3$ be the order of $s_{i_k} s_{i_{k+1}}$ in $W$. Then suppose
     \begin{align*}
      \mathbf{i} &
      = (i_1, \ldots, i_{k-1}; i_k, i_{k+1}, i_k, \ldots; i_{k+m_{i_k,i_{k+1}}+1}, \ldots, i_{\ell(w_0)}), \\
      \mathbf{i}' &
      = (i_1, \ldots, i_{k-1}; i_{k+1}, i_k, i_{k+1}, \ldots; i_{k+m_{i_k,i_{k+1}}+1}, \ldots, i_{\ell(w_0)}).
     \end{align*}
     Then define a function $R_{\mathbf{i}}^{\mathbf{i}'} : \ZZ_{\ge 0}^{\ell(w_0)} \to \ZZ_{\ge 0}^{\ell(w_0)}$ as follows:
     \begin{enumerate}
      \item
       For $k' < k$ or $k' > k + m_{i_k, i_{k+1}}$, let $R_{\mathbf{i}}^{\mathbf{i}'} (n_{\bullet})_{k'} = n_{k'}$
       
       \item
        Suppose $m_{i_k, i_{k+1}} = 2$. Then let $R_{\mathbf{i}}^{\mathbf{i}'} (n_{\bullet})_k = n_{k+1}$ and let $R(n_{\bullet})_{k+1} = n_k$.
        
       \item
        Suppose $m_{i_k, i_{k+1}} = 3$, and let $p = \min{\{n_k, n_{k+2}\}}$. Then let $R_{\mathbf{i}}^{\mathbf{i}'} (n_{\bullet})_k = n_{k+1} + n_{k+2} - p$, $R_{\mathbf{i}}^{\mathbf{i}'} (n_{\bullet})_{k+1} = p$, and $R_{\mathbf{i}}^{\mathbf{i}'} (n_{\bullet})_{k+2} = n_k + n_{k+1} - p$.
     \end{enumerate}
     
     If $n_{\bullet}$ is the $\mathbf{i}$-Lusztig datum of a stable AMV cycle $A_0$, then $R_{\mathbf{i}}^{\mathbf{i}'} (n_{\bullet})$ is the $\mathbf{i}'$-Lusztig datum of $A_0$.
    \end{Lemma}
    
    Since any two simple roots in a $\sigma$-orbit are either disconnected or connected by an edge of weight $1$ in the Dynkin diagram, we are only interested in the cases above. However, \cite{Ka10} proves the proposition in all cases. Since all reduced words for $w_0$ are related by sequences of braid moves, Kamnitzer's full proposition implies the existence of a function $R_{\mathbf{i}}^{\mathbf{i}'}$ for all pairs of reduced words $(\mathbf{i}, \mathbf{i}')$. As a consequence of the cases above, we can define $R_{\mathbf{i}}^{\mathbf{i}'}$ whenever $\mathbf{i}$ and $\mathbf{i}'$ expand a common word $\mathbf{i}_{\sigma}$. Thus we have justification for the following definition.
    \begin{DEF}[$\sigma$-invariant $\mathbf{i}$-Lusztig datum]
     Fix a reduced word $\mathbf{i}$ for $w_0$ which expands a reduced word $\mathbf{i}_{\sigma}$ for $w_{0, \sigma}$. Then an $\mathbf{i}$-Lusztig datum $n_{\bullet}$ is $\sigma$-invariant if it is constant on each $\sigma$-orbit. If $n_{\bullet}$ is a $\sigma$-invariant $\mathbf{i}$-Lusztig datum, then let $\bar{n}_{\bullet}$ be the corresponding $\mathbf{i}_{\sigma}$-Lusztig datum.
    \end{DEF}
    
    Consider Example \ref{Eg:compatible}. In this case, an $\mathbf{i}$-Lusztig datum $n_{\bullet}$ is $\sigma$-invariant if and only if $n_1 = n_2$, $n_3 = n_4 = n_5$, $n_6 = n_7$, and $n_8 = n_9 = n_{10}$. If this is the case, then $\bar{n}_{\bullet} = (n_1, n_3, n_6, n_8)$. 
    
    Note that $\mathbf{i}$ and $\sigma(\mathbf{i})$ are two different reduced words for $w_0$ expanding $\mathbf{i}_{\sigma}$, and $R_{\mathbf{i}}^{\sigma(\mathbf{i})} (n_{\bullet}) = n_{\bullet}$ if and only if $n_{\bullet}$ is $\sigma$-invariant. Thus, given a $\sigma$-compatible reduced word $\mathbf{i}$ for $w_0$, an $\mathbf{i}$-Lusztig datum $n_{\bullet}$ is $\sigma$-invariant if and only if the stable MV cycle $\overline{A_0^{\mathbf{i}} (n_{\bullet})}$ is $\sigma$-invariant.
    
    \begin{Lemma}
     \label{Lemma:Lusztig-fp}
     Let $\mathbf{i}$ be a $\sigma$-compatible reduced word for $w_0$, and let $n_{\bullet}$ be a $\sigma$-invariant $\mathbf{i}$-Lusztig datum. Then we can identify the fixed-point subvariety
     \[
      (A_0^{\mathbf{i}} (n_{\bullet}))^{\sigma} = A_0^{\mathbf{i}_{\sigma}} (\bar{n}_{\bullet}).
     \]
    \end{Lemma}
    
    \begin{proof}
     Recall the surjective map
     \[
      \mathbf{y}_{\mathbf{i}}: B(n_{\bullet}) \to A_0^{\mathbf{i}} (n_{\bullet})
     \]
     from the proof of Theorem \ref{THM:Lusztig-index}, which is defined explicitly in the proof of \cite{Ka10} Theorem 4.5. Consider also the map
     \[
      \mathbf{y}_{\mathbf{i}_{\sigma}} : B_{\sigma}(\bar{n}_{\bullet}) \to A_0^{\mathbf{i}_{\sigma}} (\bar{n}_{\bullet}),
     \]
     defined analogously for $G^{\sigma}$. It is clear from the explicit definition of $\mathbf{y}_{\mathbf{i}}$ that in the square
     \[
      \begin{tikzcd}
       B_{\sigma} (\bar{n}_{\bullet}) \ar[r, "\mathrm{diag}"] \ar[d, "\mathbf{y}_{\mathbf{i}_{\sigma}}" left] &
       B (n_{\bullet}) \ar[d, "\mathbf{y}_{\mathbf{i}}"]
        \\
       A_0^{\mathbf{i}_{\sigma}} (\bar{n}_{\bullet}) \ar[r,  "\iota"] &
       A_0^{\mathbf{i}} (n_{\bullet})
      \end{tikzcd}
     \]
     the map $\iota$ is well-defined and injective, and the square commutes, for $\sigma$-invariant $n_{\bullet}$. In particular, $A_0^{\mathbf{i}_{\sigma}} (\bar{n}_{\bullet}) \subset (A_0^{\mathbf{i}} (n_{\bullet}))^{\sigma} = A_0^{\mathbf{i}} (n_{\bullet}) \cap \GrG{G^{\sigma}}$.
     
     The reverse inclusion follows from the $\mathbf{i}_{\sigma}$- and $\mathbf{i}$-Lusztig stratifications from equation (\ref{eq:Lusztig-strat}) of the varieties 
     \[
      (X(\nu_{\ell(w_0)}, 0))^{\sigma} = 
      (S_{w_0}^{\nu_{\ell(w_0)}} \cap S_e^0) \cap \GrG{G^{\sigma}} =
      (S_{\sigma})_{w_{0, \sigma}}^{\nu_{\ell(w_0)}} \cap (S_{\sigma})_e^0
     \]
     and 
     \[
      X(\nu_{\ell(w_0)}, 0) =
      S_{w_0}^{\nu_{\ell(w_0)}} \cap S_e^0.
     \]
     Note also the semi-infinite cells $S_{w_0}^{\nu}$ and $S_e^0$ are $\sigma$-invariant, since $\sigma$-invariance of the $\mathbf{i}$-Lusztig datum $n_{\bullet}$ implies $\nu_{\ell(w_0)}$ is a $\sigma$-invariant cocharacter.
    \end{proof}
    
    \begin{Lemma}
    \label{Lemma:coweight}
    Let $n_{\bullet}$ be a $\sigma$-invariant $\mathbf{i}$-Lusztig datum, where $\mathbf{i}$ expands $\mathbf{i}_{\sigma}$. Suppose $A_0^{\mathbf{i}}(n_{\bullet})$ is of coweight $(\nu, 0)$. Then $A_0^{\mathbf{i}_{\sigma}} (\bar{n}_{\bullet})$ is also of coweight $(\nu, 0)$.
   \end{Lemma}
   
   \begin{proof}
    Let $\pi: \{1, \ldots, \ell(w_0)\} \to \{1, \ldots ,\ell(w_{0, \sigma})\}$ be the surjection of indices of the words. It is sufficient to prove that, for each $1 \le k' \le \ell(w_{0, \sigma})$, we have
    \begin{equation}
     \label{eq:betas}
     \bar{n}_{k'} \beta^{\mathbf{i}_{\sigma}, \vee}_{k'} =
     \sum_{k \in \pi^{-1}(k')} n_k \beta^{\mathbf{i}, \vee}_k.
    \end{equation}
    Then since $\bar{n}_{k'} = n_k$ for all $k \in \pi^{-1}(k')$, equation (\ref{eq:betas}) is equivalent to
    \begin{equation}
     \label{eq:betas-better}
     \beta^{\mathbf{i}_{\sigma}, \vee}_{k'} = 
     \sum_{k \in \pi^{-1} (k')} \beta^{\mathbf{i}, \vee}_k
    \end{equation}
    for all $k'$. For each $k'$ there are two possibilities. Either $\pi^{-1} (k') = \{k, k+1, k+2\}$, where $\alpha^{\vee}_{i_k} = \alpha^{\vee}_{i_{k+2}}$ and $\alpha^{\vee}_{i_{k+1}}$ are two distinct simple coroots connected by an edge of weight $1$ in the Dynkin diagram of $G$; or all simple coroots corresponding to $k \in \pi^{-1} (k')$ are pairwise disconnected. In the first case we can say the orbit is of type $A_2$, and in the second we can say it is of type $A_1 \times \cdots \times A_1$.
    
    First suppose $k'$ corresponds to an orbit of type $A_2$. Recall 
    \begin{align*}
     \beta^{\mathbf{i}, \vee}_{k + j} &
     = w^{\mathbf{i}}_{k + j} (\rho^{\vee}) - w^{\mathbf{i}}_{k + j -1} (\rho^{\vee})
     = w^{\mathbf{i}}_{k + j - 1} (s_{i_{k + j}}(\rho^{\vee}) - \rho^{\vee}) \ \text{ (for $j = 0, 1, 2$), \ and}
      \\
     \beta^{\mathbf{i}_{\sigma}, \vee}_{k'} &
     = w^{\mathbf{i}_{\sigma}}_{k'}(\rho^{\vee}_{\sigma}) - w^{\mathbf{i}_{\sigma}}_{k' - 1} (\rho^{\vee}_{\sigma})
     = w^{\mathbf{i}_{\sigma}}_{k' - 1} (s_{i_{k'}} (\rho^{\vee}_{\sigma}) - \rho^{\vee}_{\sigma}).
    \end{align*}
    Then the right hand side of (\ref{eq:betas-better}) is
    \begin{align}
     \nonumber
     \beta^{\mathbf{i}, \vee}_k + \beta^{\mathbf{i}, \vee}_{k+1} + \beta^{\mathbf{i}, \vee}_{k + 2} &
     = w^{\mathbf{i}}_{k+2} (\rho^{\vee}) - w^{\mathbf{i}}_{k - 1} (\rho^{\vee}) 
      \\
     \nonumber
     &
     = w^{\mathbf{i}}_{k-1} (s_{i_k} s_{i_{k+1}} s_{i_{k+2}} (\rho^{\vee}) - \rho^{\vee})
      \\
     \label{eq:rho}
     &
     = w^{\mathbf{i}_{\sigma}}_{k' - 1} (s_{i_{k'}} (\rho^{\vee}_{\sigma}) - \rho^{\vee}_{\sigma}).
    \end{align}
    The equality of line (\ref{eq:rho}) follows from Propositions \ref{Prop:Weyl} and \ref{Prop:rho}: $s_{i_{k'}} = s_{i_k} s_{i_{k+1}} s_{i_{k+2}}$ is the longest element of the Coxeter subgroup generated by letters in the orbit $\eta$, and so by induction $w^{\mathbf{i}}_{k - 1} = w^{\mathbf{i}_{\sigma}}_{k' - 1}$.
    
    The case $k'$ corresponds to an orbit of type $A_1 \times \cdots \times A_1$ is similar, except the orbit can be of any order,
    and each index in the orbit appears only once. Then in this case $s_{i_{k'}} = s_{i_k} s_{i_{k+1}} \cdots s_{i_{k + |\eta| - 1}}$, and so equation (\ref{eq:betas-better}) holds.
   \end{proof}
   
  \section{The bijection for AMV cycles}
   \label{S:Gathering}
   
   \begin{THM}
    \label{THM:bijection}
    Let $(\lambda, \mu)$ be a coweight with $\mu$ dominant, $\lambda \le \mu$, and both fixed by $\sigma$. Then taking $\sigma$-fixed points induces a bijection between $\sigma$-invariant MV cycles in $\GrG{G}$ of coweight and $(\lambda, \mu)$ and MV cycles in $\GrG{G^{\sigma}}$ of coweight $(\lambda, \mu)$.
   \end{THM}
   
   \begin{proof}
    First  I will establish the bijection between AMV cycles. The action of $-\mu$ on the set of AMV cycles gives a bijection between AMV cycles of coweight $(\lambda, \mu)$ and stable AMV cycles of coweight $(\lambda - \mu, 0)$.
    
    Fix a $\sigma$-compatible reduced word $\mathbf{i}$ for $w_0$ expanding the reduced word $\mathbf{i}_{\sigma}$ for $w_{0, \sigma}$. Theorem \ref{THM:Lusztig-index} gives a bijection between $\mathbf{i}$-Lusztig data and all stable AMV cycles, and restricts to a bijection between those of coweight $(\lambda - \mu, 0)$. And by Lemma \ref{Lemma:Lusztig-fp}, there is a bijection between $\sigma$-invariant $\mathbf{i}$-Lusztig strata of coweight $(\lambda - \mu, 0)$ and $\mathbf{i}_{\sigma}$-Lusztig strata of coweight $(\lambda - \mu, 0)$, given by taking fixed points.
    
    Then by composing the corresponding bijection between $\mathbf{i}_{\sigma}$-Lusztig strata of coweight $(\lambda - \mu, 0)$ and AMV cycles of coweight $(\lambda, \mu)$, we have the desired bijection on the level of AMV cycles.
   
    It remains to see that the bijection on AMV cycles restricts to one on MV cycles. Recall Proposition \ref{Prop:Polytope}: MV cycles of coweight $(\lambda, \mu)$ correspond bijectively to those AMV cycles of coweight $(\lambda, \mu)$ contained in $\overline{\GrG{G}^{\mu}}$. Clearly, if $A$ is an AMV cycle and $A^{\sigma} \not\subset \overline{\GrG{G^{\sigma}}^{\mu}}$, then $A \not \subset \overline{\GrG{G}^{\mu}}$. However, the converse is more subtle.
   
    Let $A$ be a $\sigma$-invariant AMV cycle as in the hypotheses, and suppose $A \not \subset \overline{\GrG{G}^{\mu}}$; i.e. $A$ is not an MV cycle. Let $(\nu_w)_{w \in W}$ be the GGMS datum of $A$, the $W$-indexed sequence of vertices in the moment polytope $P_A$. Then the GGMS stratum of $A$ is
    \[
     GGMS(A) :=
     \bigcap_{w \in W} S_w^{\nu_w}.
    \]
    Note that $\overline{\GrG{G}^{\mu}}$ is an AMV cycle of coweight $(-\mu, \mu)$, and thus the closure of its GGMS stratum:
    \[
     \overline{\GrG{G}^{\mu}} = 
     \overline{\bigcap_{w \in W} S_w^{w(\mu)}}
    \]
    Now if $A \not\subset \overline{\GrG{G}^{\mu}}$, then its GGMS stratum does not intersect with $\overline{\GrG{G}^{\mu}}$ at all. Indeed, suppose there exists a complex point
    \[
     p \in 
     \left(GGMS(A) \cap \overline{\GrG{G}^{\mu}}\right) (\CC) =
     \left(\bigcap_{w \in W} S_w^{\nu_w}(\CC)\right) \cap \left(\overline{\bigcap_{w \in W} S_w^{w(\mu)}}(\CC)\right) \subset
     \bigcap_{w \in W} \left(S_w^{\nu_w}(\CC) \cap \overline{S_w^{w(\mu)}}(\CC)\right).
    \]
    Then for each $w \in W$ the closure relations for semi-infinite cells imply $\nu_w \le_w w(\mu)$, so $P_A \subset \mathrm{Conv}(W \cdot \mu)$. By Theorem \ref{THM:Polytope}, this implies $A$ is in fact an MV cycle, contrary to assumption.
   
    By Lemma \ref{Lemma:Polytope}, $GGMS(A)$ is dense in $A$. So if $(A \cap \overline{\GrG{G}^{\mu}})^{\sigma}$ were dense in $A^{\sigma}$, there would have to be some complex point
    \[
     p \in
     \left(GGMS(A)^{\sigma} \cap \overline{\GrG{G^{\sigma}}^{\mu}}\right) (\CC) \subset
     \left(GGMS(A)\cap \overline{\GrG{G}^{\mu}}\right) (\CC),
    \]
    but the last intersection is empty.
   
    As a result, the map
    \[
     \begin{tikzcd}
      \left\{\begin{array}{c} A \subset \GrG{G} \end{array} \middle| \begin{array}{c} \text{$\sigma$-invariant MV cycles}\\ \text{of coweight $(\lambda, \mu)$}\end{array}\right\} \ar[r] &
      \left\{\begin{array}{c} A \subset \GrG{G^{\sigma}} \end{array} \middle| \begin{array}{c} \text{MV cycles of} \\  \text{coweight $(\lambda, \mu)$} \end{array} \right\}
       \\
      A \ar[r, mapsto] &
      \overline{A^{\sigma}} \cap (S_{\sigma})_{w_{0,\sigma}}^{\lambda}
     \end{tikzcd}
    \]
    is a bijection.
   \end{proof}
  
  \section{Eigenvalues}
   \label{S:Eigenvalues}
   
     I would like to prove that equation (\ref{eq:main}) holds for all $\sigma$-invariant $(\lambda, \mu)$. It is known that the number of $\sigma$-invariant MV cycles of coweight $(\lambda, \mu)$ in $\GrG{G}$ is the same as the number of MV cycles of coweight $(\lambda, \mu)$ in $\GrG{G^{\sigma}}$, and that the bijection is given by taking the $\sigma$-fixed points. From a naive understanding of the geometric Satake equivalence, it is thus clear that there \textit{is} a linear map $\sigma': V_{\mu}(\lambda) \to V_{\mu}(\lambda)$ such that equation (\ref{eq:main}) holds, replacing $\sigma$ with $\sigma'$. Indeed, one may simply choose a basis $\{e_A\}$ for $V_{\mu}(\lambda)$ indexed by MV cycles $A$, and let $\sigma'$ be any map such that $\sigma'(e_A)$ is a scalar multiple if $e_{\sigma(A)}$ for all $A$ and $\sigma'(e_A) = e_A$ for $\sigma$-invariant $A$. However, it is not clear at this point that the map constructed from the action of $\sigma$ on $\widehat{G}$ shares these properties. We need to take a more careful look at the geometric Satake equivalence to see that there is indeed a basis $\{e_A\}$ for which $\sigma$ satisfies these properties.
  
  But first, we need to more carefully define the operator we are considering, as well as construct an alternative operator for comparison. In the end, we will need to make some (limited) choice to identify the operators on $V_{\mu}(\lambda)$, so from here I will start decorating them to keep them distinct. 
  
  On the one hand, we have constructed $\widehat{G}$ as the group dual to $G$ using the root datum, and given it an arbitrary pinning (although there is a canonical choice one could make, to me it is arbitrary since I will make an identification that is not canonical in any case). This pinning uniquely determines an automorphism of $\widehat{G}$, which I will now call $\hat{\sigma}$, preserving $\widehat{B}$ and $\widehat{T}$, and compatible with the root homomorphisms $x_{\alpha^{\vee}}$. It is now straightforward to construct an action of $\hat{\sigma}$ on irreducible highest-weight representations $V_{\mu}$ of $\widehat{G}$ when $\mu$ is $\hat{\sigma}$-invariant. Since $\hat{\sigma}^*$ is a tensor auto-equivalence on $\Rep{\CC}{\widehat{G}}$, we have $\hat{\sigma}^*V_{\mu}$ is an irreducible highest-weight representation. Furthermore, given a vector $v \in V_{\mu}$ of weight $\lambda$ and $t \in \widehat{T}$, we have $\hat{\sigma}(t) \cdot v = \sigma^{-1}(\lambda)(t)v$, so as an element of $\hat{\sigma}^*V_{\mu}$, $v$ has weight $\sigma^{-1}(\lambda)$. In particular, the highest weight of $\hat{\sigma}^*V_{\mu}$ is $\sigma^{-1}(\mu)$. In the case $\mu$ is $\sigma$-invariant, we thus have $\hat{\sigma}^*V_{\mu} \cong V_{\mu}$. By Schur's lemma, there is up to scalar a unique isomorphism of representations $\hat{\sigma}^*V_{\mu} \to V_{\mu}$. However, since $\hat{\sigma}^*V_{\mu}$ and $V_{\mu}$ share underlying vector spaces, we may canonically choose the isomorphism which fixes the highest-weight line $V_{\mu}(\mu)$ pointwise. This automorphism on the underlying vector space is the action of $\sigma$ on $V_{\mu}$, and will be known as $\hat{\sigma}_{\mu}: V_{\mu} \to V_{\mu}$. It is difficult, from this construction, to directly deduce precise eigenvalues of $\hat{\sigma}_{\mu}$ for vectors outside of the highest-weight line $V_{\mu}(\mu)$ (and other extreme-weight lines $V_{\mu}(w(\mu))$).
  
  On the other hand, we have a construction of the dual group of $G$ using the Tannakian formalism: $\widetilde{G}$ is the group of fiber functor automorphisms of $P_{L^+G}(\GrG{G})$. I will construct a linear isomorphism $\tilde{\sigma}_{\mu}: \HH^{\bullet}(\GrG{G}, IC_{\mu}) \to \HH^{\bullet}(\GrG{G}, IC_{\mu})$ for $\sigma$-invariant cocharacters $\mu$ of $G$, and identify $\widetilde{G} \cong \widehat{G}$ in such a way that (using uniqueness from Schur's lemma) the automorphisms $\hat{\sigma}_{\mu}$ and $\tilde{\sigma}_{\mu}$ must be equal. However, the identification of $\widetilde{G}$ with $\widehat{G}$ is non-canonical, and in particular is sensitive to the pinning on $G$. In order to identify $\widetilde{G}$ with $\widehat{G}$, I will construct a pinning on $\widetilde{G}$ that is preserved by an automorphism $\tilde{\sigma}$ of $\widetilde{G}$. 
  
  The advantage of considering these "tilde" constructions is that, as an induced map on cohomology groups, it will be much more straightforward to prove that $\tilde{\sigma}_{\mu}$ fixes all basis vectors corresponding to $\sigma$-invariant MV cycles, and thus satisfies equation (\ref{eq:main}).
  
  I will henceforth refer to the functor $\HH^{\bullet}(\GrG{G}, -)$ as $F$, as in ``fiber.'' Similarly, for cocharacters $\mu$ and $\lambda$, let $F_{\lambda} IC_{\mu}$ be the group $H_c^{-2\langle{\rho, \mu - \lambda}\rangle}(S_{w_0}^{\lambda} \cap \GrG{G}^{\mu}, \CC)$.
  
  The map $\tilde{\sigma}_{\mu}$ will be constructed from an identification $IC_{\mu} \cong \sigma^! IC_{\mu}$. (Note that, since $\sigma$ is an \'etale morphism of varieties, $\sigma^*\Aa$ is canonically isomorphic to $\sigma^!\Aa$ for all perverse sheaves $\Aa$ on $\GrG{G}$. Similarly, since $\sigma$ is proper, $\sigma_!\Aa$ is canonically isomorphic to $\sigma_*\Aa$.) Suppose we have an isomorphism $\phi_{\mu}: IC_{\mu} \to \sigma^! IC_{\mu}$. Then, using the counit $\epsilon$ of the adjunction $(\sigma_!, \sigma^!)$ we have an isomorphism
  \[
   \begin{tikzcd}
    \sigma_! IC_{\mu} \ar[r, "\sigma_! \phi_{\mu}"] &
    \sigma_! \sigma^! IC_{\mu} \ar[r, "\epsilon"] &
    IC_{\mu}
   \end{tikzcd}
  \]
  Since $\GrG{G}$ is an ind-proper ind-scheme, $F$ is canonically isomorphic to a direct sum of compactly supported cohomology functors, implying $F \, IC_{\mu} \cong F \sigma_! IC_{\mu}$. Thus, we can compose arrows in the following (commutative) diagram
  \[
   \begin{tikzcd}
    F \, IC_{\mu} \ar[r, "F \phi_{\mu}"] \ar[d, "can"] &
    F \sigma^! IC_{\mu} \ar[d, "can"] &
     \\
    F \sigma_! IC_{\mu} \ar[r, "F \sigma_! \phi_{\mu}"] &
    F \sigma_! \sigma^! IC_{\mu} \ar[r, "F \epsilon"]&
    F \, IC_{\mu}
   \end{tikzcd}
  \]
  to produce a map $\tilde{\sigma}_{\mu} := F \epsilon \circ can \circ F \phi_{\mu}$, given an isomorphism $\phi_{\mu}$.
  
  Similarly to the construction of $\hat{\sigma}_{\mu}$, we can construct $\phi_{\mu}$ canonically. For all $x$ in the smooth locus $\GrG{G}^{\mu}$ of $\overline{\GrG{G}^{\mu}}$, the stalks of both $IC_{\mu}$ and $\sigma^! IC_{\mu}$ at $x$ are one-dimensional and concentrated in degree $-2\langle{\rho, \mu}\rangle$. Furthermore, the basepoint $\varpi^{\mu} x_0$ is preserved by $\sigma$ if $\mu$ is. So let $\phi_{\mu}$ be the unique isomorphism $IC_{\mu} \to \sigma^! IC_{\mu}$ that restricts to the identity on the stalk at $\varpi^{\mu} x_0$.
  
  \begin{Lemma}
  Suppose $\mu$ is a $\sigma$-invariant dominant cocharacter. As constructed above, the map $\tilde{\sigma}_{\mu}$ is a direct sum of natural maps on top-dimensional cohomology groups with compact support, with constant coefficients, induced by the morphisms of varieties $\sigma: S_{w_0}^{\lambda} \cap \GrG{G}^{\mu} \to S_{w_0}^{\sigma(\lambda)} \cap \GrG{G}^{\mu}$. As a result, for all irreducible components $A \subset S_{w_0}^{\lambda} \cap \GrG{G}^{\mu}$, the fundamental class 
  \[
   [A] \in H_c^{2\langle{\rho, \mu - \lambda}\rangle}(S_{w_0}^{\lambda} \cap \GrG{G}^{\mu}, \CC)
  \]
  satisfies $\tilde{\sigma}_{\mu}([A]) \in \CC[\sigma(A)]$. And in particular, if $\sigma(A) = A$, then $\tilde{\sigma}_{\mu}([A]) = [A]$ exactly.
  \end{Lemma}
  
  \begin{proof}
   Recall by the geometric Satake correspondence, specifically \ref{THM:MV} \ref{THM:MV:iii} and \ref{THM:MV:iv}, that there is a natural, canonical isomorphism
   \[
    GSE_{\mu}:
    F \, IC_{\mu} \overset{\sim}{\to}
    \bigoplus_{\lambda \in X_*(T)} H_c^{-2\langle{\rho, \lambda}\rangle}(S_{w_0}^{\lambda}, IC_{\mu}) \overset{\sim}{\to}
    \bigoplus_{\lambda \in Wt(\mu)} H_c^{2\langle{\rho, \mu - \lambda}\rangle} (S_{w_0}^{\lambda} \cap \GrG{G}^{\mu}, \CC).
   \]
   So the goal will be to show that a natural map on cohomology groups induced by the morphism $\sigma: S_{w_0}^{\lambda} \cap \GrG{G}^{\mu} \to S_{w_0}^{\sigma(\lambda)} \cap \GrG{G}^{\mu}$ commutes with $GSE_{\mu}$ and fixes pointwise the cohomology groups corresponding to $\sigma$-invariant MV cycles.
   
   Cohomology with compact support is a covariant functor. That is, given a map of sheaves $\phi: \Ff \to \Gg$ on a variety $Y$, we have for each $i$ a map on cohomology with compact support
   \[
    H_c^i (\phi): 
    H_c^i (Y, \Ff) \to
    H_c^ i(Y, \Gg)
   \]
   Given a morphism of varieties $f: X \to Y$, we have naturally induced such a map on sheaves $\epsilon: f_!f^!\CC_Y \to \CC_Y$, where $\epsilon$ is the counit of the adjunction $(f_!, f^!)$. Furthermore, $H_c^i(Y, f_!f^! \CC_Y)$ is canonically isomorphic to $H_c^i(X, f^! \CC_Y)$. If $f$ is an isomorphism of varieties, this is further isomorphic to $H_c^i(X, \CC_X)$ as $f^! \CC_Y \cong \CC_X$. Thus, after choosing an appropriate isomorphism $\CC_X \to f^! \CC_Y$, we have a map
   \[
    \mathrm{Tr}^i_f: H_c^i(X, \CC) \to H_c^i(Y, \CC).
   \]
   This map is constructed in \cite[0GJY]{SP22} and referred to as the trace map of $f$.
   
   Let $i = 2\langle{\rho, \mu - \lambda}\rangle$, where $\mu$ is a $\sigma$-invariant cocharacter and $\lambda \in Wt(\mu)$. Then the cohomology group $H_c^i(S_{w_0}^{\lambda} \cap \GrG{G}^{\mu}, \CC)$ is spanned by fundamental classes $[A]$, where $A$ is an MV cycle of coweight $(\lambda, \mu)$, and $H_c^i(S_{w_0}^{\sigma(\lambda)} \cap \GrG{G}^{\mu}, \CC)$ is spanned by fundamental classes $[\sigma(A)]$. Looking on stalks, we see that $\mathrm{Tr}^i_{\sigma} ([A])$ is a scalar multiple of $[\sigma(A)]$. Supposing $\lambda$ is $\sigma$-invariant and choosing the normalization $\CC \to \sigma^!\CC$ corresponding to that for $\tilde{\sigma}_{\mu}$, which is identity on stalks preserved by $\sigma$, we get that in the case $\sigma(A) = A$, $\mathrm{Tr}^i_{\sigma}([A]) = [A]$ exactly.
   
   By the similarity in construction of the two horizontal maps, the following diagram commutes:
   \[
    \begin{tikzcd}
     F \, IC_{\mu} \ar[r, "\tilde{\sigma}_{\mu}"] \ar[d, "GSE_{\mu}"] &
     F \, IC_{\mu} \ar[d, "GSE_{\mu}"] &
      \\
     \displaystyle \bigoplus_{\lambda \in Wt(\mu)} H_c^{2\langle{\rho, \mu - \lambda}\rangle}(S_{w_0}^{\lambda} \cap \GrG{G}, \CC) \ar[r, "\mathrm{Tr}^i_{\sigma}"] &
     \displaystyle \bigoplus_{\lambda \in Wt(\mu)} H_c^{2\langle{\rho, \mu - \lambda}\rangle}(S_{w_0}^{\sigma(\lambda)} \cap \GrG{G}, \CC)
    \end{tikzcd}
   \]
   Therefore $\tilde{\sigma}_{\mu}$ acts on fundamental classes as expected.
  \end{proof}
  
  Note that $\tilde{\sigma}_{\mu}$ thus satisfies (\ref{eq:main}).

  In order to identify the actions $\tilde{\sigma}$ and $\hat{\sigma}$, we extend the construction of $\tilde{\sigma}_{\mu}$ to cases where $\mu$ is not $\sigma$-invariant, constructing an automorphism $\tilde{\sigma}: \widetilde{G} \to \widetilde{G}$. Given such an automorphism $\tilde{\sigma}$, we can verify that the linear map $\tilde{\sigma}^*V_{\mu} \to V_{\mu}$ on induced on $\sigma$-invariant highest-weight representations of $\widetilde{G}$ by the automorphism $\tilde{\sigma}$ is equal to the linear operator $\tilde{\sigma}_{\mu}$ defined above. Then, in order to prove that $\hat{\sigma}_{\mu}$ satisfies equation (\ref{eq:main}), it will be sufficient to verify that $\widetilde{G}$ may be identified with $\widehat{G}$ in such a way that the automorphisms $\tilde{\sigma}$ and $\hat{\sigma}$ commute with the identification $\widetilde{G} \to \widehat{G}$.
  
  Consider the following proposition, an immediate consequence of definitions in the first chapter of \cite{DM89}:
  
  \begin{Prop}
   \label{Prop:auto}
   Let $\widehat{G}$ be a reductive group, and let $F: \Rep{\CC}{\widehat{G}} \to \Vect{\CC}$ be the natural fiber functor. Given a tensor auto-equivalence $T: (\Rep{\CC}{\widehat{G}}, \otimes) \to (\Rep{\CC}{\widehat{G}}, \otimes)$ and an isomorphism of fiber functors $\phi: FT \to F$, there is a corresponding automorphism $\tau$ of $\widehat{G}$, given by $\tau(g) = \phi \circ g \circ \phi^{-1}$. In the other direction, an automorphism $\tau$ can be used to construct such a pair, using the functor $\tau^*(\rho, V) = (\rho \circ \tau, V)$, and the isomorphism $F\tau^* \to F$ which takes identity on objects in $\Vect{\CC}$.
   \[
    \begin{array}{ccc}
     \Aut{\parens{\widehat{G}}} &
     \begin{array}{c} \longrightarrow \\ \longleftarrow \end{array}  &
     \left\{\begin{array}{c}\text{tensor auto-equivalences with} \\ \text{an isomorphism of fiber functors} \\ T: (\Rep{\CC}{\widehat{G}}, \otimes) \to (\Rep{\CC}{\widehat{G}}, \otimes)  \\ \phi: FT \to F \end{array}\right\}
  	  \\
     \tau &
     \longmapsto &
     (\tau^*, \id{})
      \\
     (g \mapsto \phi \circ g \circ \phi^{-1}) &
     \longmapsfrom &
     (T, \phi)
    \end{array}
   \]
  \end{Prop}
  
  We could impose an equivalence relation on the right hand side above and tautologically the pairs $(T, \phi)$ modulo equivalence would be in bijection with automorphisms. All we need is a sufficient condition for construction of an automorphism $\tilde{\sigma}: \widetilde{G} \to \widetilde{G}$.
  
  We already have a tensor auto-equivalence of functors $\sigma^!: P_{L^+G}(\GrG{G}) \to P_{L^+G}(\GrG{G})$, so we want an isomorphism of fiber functors $\phi: F \sigma^! \to F$. As before, we can use the counit $\epsilon: \sigma_! \sigma^! IC_{\mu} \to IC_{\mu}$.
  \[
   \begin{tikzcd}
    F \sigma^! IC_{\mu} \ar[r, "can"] &
    F \sigma_!\sigma^! IC_{\mu} \ar[r, "F \epsilon"] &
    F \, IC_{\mu}
   \end{tikzcd}
  \]
  So we let $\phi = F \epsilon \circ can$; then $\tilde{\sigma}_{\mu} (g) = \phi \circ g \circ \phi^{-1}$, as needed.
  
  Finally, we need to compare $\tilde{\sigma}$ with $\hat{\sigma}$. Consider that, for $\sigma$-invariant $\mu$, $\tilde{\sigma}$ satisfies a commutative diagram
  \begin{equation}
    \label{eq:g-square}
    \begin{tikzcd}
     F \, IC_{\mu} \ar[r, "\tilde{\sigma}_{\mu}"] \ar[d, "g" left] &
     F \, IC_{\mu} \ar[d, "\tilde{\sigma}(g)"] &&
     F \sigma^! IC_{\mu} \ar[ll, "F \epsilon \circ can" above] \ar[d, "g"]
      \\
     F \, IC_{\mu} \ar[r, "\tilde{\sigma}_{\mu}"] &
     F \, IC_{\mu} &&
     F \sigma^! IC_{\mu} \ar[ll, "F \epsilon \circ can" above]
    \end{tikzcd}
   \end{equation}
   In fact $\tilde{\sigma}_{\mu}$ factors through $F \sigma^! IC_{\mu}$, implying the linear operator on $F \, IC_{\mu}$ induced by the automorphism $\tilde{\sigma}$ is exactly $\tilde{\sigma}_{\mu}$.
  
  Now we turn our attention to the question of identifying $\widetilde{G}$ and $\widehat{G}$. They are by GSE \ref{THM:MV} \ref{THM:MV:v} abstractly isomorphic. Since a pinning-preserving automorphism is determined uniquely by its action on the Dynkin diagram and the pinning it preserves, it is sufficient to prove that there is \textit{some} pinning preserved by $\tilde{\sigma}$; any choice of such a preserved pinning will imply $\tilde{\sigma} = \hat{\sigma}$ after identification $\widetilde{G} \cong \widehat{G}$. Note that the pinning preserved by $\tilde{\sigma}$ depends on the isomorphism $\sigma: \GrG{G} \to \GrG{G}$, which ultimately depends on the pinning on $G$. It is for this reason that the identification $\widetilde{G} \cong \widehat{G}$ is not canonical.
   
   To begin with, consider that $\widetilde{G}$ has a natural choice of maximal torus and Borel $\widetilde{T} \subset \widetilde{B} \subset \widetilde{G}$, implying that identification with $\widehat{G}$ may only vary by conjugation by $\widetilde{T}$. Indeed, let $\widetilde{T} \subset \widetilde{G}$ consist of those fiber functor automorphisms that preserve weight spaces of all representations, and $\widetilde{B} \subset \widetilde{G}$ consist of those fiber functor automorphisms that preserve the positive cone of weight spaces. In particular, $\widetilde{B}$ is generated by $\widetilde{T}$ and $\widetilde{U}_{\alpha^{\vee}}$ for $\alpha^{\vee} \in \Phi^{\vee, +}$, where we have $\widetilde{U}_{\alpha^{\vee}}$ defined by the property
   \begin{equation}
    \label{eq:Ualpha}
    g(F_{\lambda} IC_{\mu}) \subset
    F_{\lambda + \alpha^{\vee}} IC_{\mu}.
   \end{equation}

   Since the action of $\tilde{\sigma}$ on $V_{\mu}$ permutes weight spaces according to the action of $\sigma$ on $X_*(T)$, the group automorphism $\tilde{\sigma}$ preserves $\widetilde{T}$ and $\widetilde{B}$. And $\widetilde{T}$ can be identified with $\widehat{T}$, such that $X_*(\widetilde{T}) = X_*(\widehat{T})$. In particular, $\tilde{\sigma}|_{\widetilde{T}} = \hat{\sigma}|_{\widehat{T}}$. Similarly, $\tilde{\sigma}$ permutes root subgroups according to the action of $\sigma$ on $\Phi^{\vee}$. Indeed, equations (\ref{eq:g-square}) and (\ref{eq:Ualpha}) imply that $g \in \widetilde{U}_{\alpha^{\vee}}$ maps to $\tilde{\sigma}(g) \in \widetilde{U}_{\sigma(\alpha^{\vee})}$.
   
    It remains to be seen that there exists a pinning $\{x_{\alpha^{\vee}_i}\}_{\alpha^{\vee}_i \in \Pi^{\vee}}$ preserved by $\tilde{\sigma}$. However, if such a pinning exists, in each $\sigma$-orbit $\eta \subset \Pi^{\vee}$, one root homomorphism $x_{\alpha^{\vee}_i}: \GG_a \to \widetilde{U}_{\alpha^{\vee}_i}$ may be determined arbitrarily, with $x_{\sigma(\alpha^{\vee}_i)} := \tilde{\sigma} \circ x_{\alpha^{\vee}_i}$. In the case $\sigma$ acts freely on a simple root, i.e. $|\eta|$ is equal to the order of $\sigma$, there is no further obstruction: an arbitrary choice of pinning for one $\alpha^{\vee}_i \in \eta$ will determine a root homomorphism respected by $\tilde{\sigma}$ for all $\alpha^{\vee}_j \in \eta$. However, if $\sigma^n(\alpha^{\vee}_i) = \alpha^{\vee}_i$ for some $n$ less than the order of $\sigma$, we need to know $\sigma^n$ preserves $\widetilde{U}_{\alpha^{\vee}_i}$ pointwise. It is sufficient to consider the case $n=1$, as in the following lemma of Hong.
   
   \begin{Lemma}[\cite{Ho09} Lemma 4.3]
    \label{Lemma:Lie}
    Suppose some simple $\alpha^{\vee}_i \in \Pi^{\vee}$ is $\sigma$-invariant, i.e. $\sigma(\alpha^{\vee}_i) = \alpha^{\vee}_i$. Then $\tilde{\sigma}$ fixes $\widetilde{U}_{\alpha^{\vee}_i}$ pointwise.
   \end{Lemma}
   
   \begin{proof}
    We may assume $\widetilde{G}$ is semisimple and almost simple. Indeed, $\widetilde{G}^{der} = \widetilde{G}_1 \times \cdots \times \widetilde{G}_m$ where each $\widetilde{G}_j$ is semisimple and almost simple, and the inclusion $\widetilde{U}_{\alpha^{\vee}_i} \hookrightarrow \widetilde{G}$ factors through some $\widetilde{G}_j \to \widetilde{G}^{der} \to \widetilde{G}$. If $\sigma$ preserves $\alpha^{\vee}_i$, this inclusion commutes with $\tilde{\sigma}$. So suppose $\widetilde{G} = \widetilde{G}_j$.
    
    We compare two different actions of $\sigma$ on $\tilde{\frakg}$, the Lie algebra of $\widetilde{G}$. The first, $d\tilde{\sigma}$, comes from differentiating the automorphism $\tilde{\sigma}$ at the identity. The second, $\tilde{\sigma}_{\gamma^{\vee}}$, comes from viewing $\tilde{\frakg}$ as the representation $V_{\gamma^{\vee}}$, where $\gamma^{\vee} \in \Phi^{\vee}$ is the highest coroot. As noted earlier, $\tilde{\sigma}_{\gamma^{\vee}}$ fixes the weight space $F_{\alpha^{\vee}_i} IC_{\gamma^{\vee}} = \tilde{\frakg}_{\alpha^{\vee}_i}$ pointwise, so it is sufficient to prove that $d\tilde{\sigma} = \tilde{\sigma}_{\gamma^{\vee}}$, as tangent space isomorphisms, implying the automorphism $\tilde{\sigma}$ fixes $\widetilde{U}_{\alpha^{\vee}_i}$ pointwise.
    
    For each $\alpha^{\vee} \in \Phi^{\vee}$, let $e_{\alpha^{\vee}}$ be the fundamental class of the (unique) MV cycle of coweight $(\alpha^{\vee}, \gamma^{\vee})$ in $\GrG{G}$. Note that for $\sigma$-invariant $\alpha^{\vee}$, we have $\tilde{\sigma}_{\gamma^{\vee}} (e_{\alpha^{\vee}}) = e_{\alpha^{\vee}}$. We can also identify $\tilde{\frakh} = Lie(\widetilde{T})$ with $X_*(\widetilde{T}) \otimes_{\ZZ} \CC$, to understand the action of $d\tilde{\sigma}$ on $\tilde{\frakh}$.
    
    Schur's lemma implies that the two maps may only differ by a constant scalar: let $d\tilde{\sigma} = c \cdot \tilde{\sigma}_{\gamma^{\vee}}$. Furthermore, by commutativity of diagram (\ref{eq:g-square}), differentiating the adjoint action of $\widetilde{G}$ on $\tilde{\frakg}$, we have $\tilde{\sigma}_{\gamma^{\vee}} ([a, b]) = [d\tilde{\sigma}(a),\tilde{\sigma}_{\gamma^{\vee}} (b)]$.
    
    The highest root $\gamma \in \Phi$ is $\sigma$-invariant, so $\tilde{\sigma}_{\gamma^{\vee}}$ fixes the image of $\gamma: \GG_m \to \widetilde{T}$ pointwise and $d \tilde{\sigma}$ fixes $\CC \cdot \gamma \subset \tilde{\frakh}$ as well. So $d\tilde{\sigma}([e_{\gamma^{\vee}}, e_{-\gamma^{\vee}}]) = [e_{\gamma^{\vee}}, e_{-\gamma^{\vee}}] \in \CC \cdot \gamma$. Since $d\tilde{\sigma}$ is a Lie algebra homomorphism, we also have 
    \[
     d\tilde{\sigma}([e_{\gamma^{\vee}},e_{-\gamma^{\vee}}]) = 
     [d\tilde{\sigma}(e_{\gamma^{\vee}}),d\tilde{\sigma}(e_{-\gamma^{\vee}})] = 
     c^2 \cdot [\tilde{\sigma}_{\gamma^{\vee}} (e_{\gamma^{\vee}}), \tilde{\sigma}_{\gamma^{\vee}} (e_{-\gamma^{\vee}})] = 
     c^2 \cdot [e_{\gamma^{\vee}}, e_{-\gamma^{\vee}}].
    \]
    The last equality holds by $\sigma$-invariance of $\gamma^{\vee}$ and $-\gamma^{\vee}$. And so $c^2 = 1$, and we must have $c = \pm 1$.
    
    Now by comparing trace of $\tilde{\sigma}_{\gamma^{\vee}}$ and $d\tilde{\sigma}$ on $\tilde{\frakh}$, we see $c \ne -1$. In particular, $\tilde{\sigma}_{\gamma^{\vee}}$ preserves $\tilde{\frakh} = F_0 IC_{\gamma^{\vee}}$, so equation (\ref{eq:main}) implies that $\trres{\tilde{\sigma}}{\tilde{\frakh}} \ge 0$. Similarly, since $\sigma$ acts on $X^*(T) = X_*(\widetilde{T})$ by permutation of characters forming a basis, we have $\trres{d\tilde{\sigma}}{X_*(\widetilde{T}) \otimes_{\ZZ} K} \ge 0$ as well. But by assumption, there is a $\sigma$-invariant simple coroot $\alpha^{\vee}_i$, and so $\trres{d\tilde{\sigma}}{X_*(\widetilde{T}) \otimes_{\ZZ} K} > 0$. So the ratio of those two traces, $c = \tr{(d\tilde{\sigma})}/\tr{(\tilde{\sigma}_{\gamma^{\vee}})}$, must be nonnegative, hence $c=1$.
   \end{proof}
    
  \section{Root data and an application}
   \label{S:Haines}
   
   Here I will be explicit about the root datum and pinning of $\widehat{G}$ and $\widehat{G^{\sigma, \circ}}$, both in terms of the pinned root datum tuple, and more simply the root system. I will also prove Theorem 7.7 of \cite{Ha18}, which was originally stated without proof. 
   
   Suppose $G$ is a quasi-split, connected, reductive group over a non-Archimedean local field $F$. Then let $\Sigma$ be the root system of $G$ and let $\breve\Sigma$ be the \'echelonnage root system of $G_{\breve{F}}$, as defined in \cite{Ha18}. By Corollary 5.3 in \cite{Ha18}, $\breve{\Sigma}^\vee$ is the root system for $\widehat{G}^{I, \circ}$, where $I$ is the inertia group of $F$. Then, in light of Theorem \ref{THM:main}, to prove Theorem 7.7 of \cite{Ha18} it is sufficient to prove that $N'_\tau(\breve \Sigma^\vee)$ is equal to the set of roots of $\widehat{\widehat{\widehat{G}^{I, \circ}}^{\tau, \circ}}$.
   
   We will work with pinned root data. According to the classification of connected reductive groups (see, for instance, \cite{Sp79}), a group $H$ over an algebraically closed field is determined up to isomorphism by its root datum: a quadruple $(X, \Phi, X^{\vee}, \Phi^{\vee})$, where $X$ and $X^{\vee}$ are dual, finitely generated, free abelian groups; and $\Phi \subset X$ and $\Phi^{\vee} \subset X^{\vee}$ are dual reduced root systems. The group $H$ is determined up to inner automorphism if the root datum is associated to a particular choice of maximal torus $T \subset H$. If the root datum is based, or endowed with a system of simple roots and coroots $\Pi \subset \Phi$ and $\Pi^{\vee} \subset \Phi^{\vee}$ corresponding to a choice of Borel $T \subset B \subset H$, then the datum determines $H$ up to inner automorphism by an element of $T$. Since the systems of roots and coroots can be constructed from the quadruple $(X, \Pi, X^{\vee}, \Pi^{\vee})$, there is no need to give the entire $6$-tuple. Finally, $H$ is determined up to unique automorphism by a pinning: a collection of root homomorphisms $x_{\alpha_i}: \GG_a \to SL_2 \to H$ for $\alpha_i \in \Pi$.
   
   When we say $\sigma$ preserves a pinning of a connected reductive group $H$ over an algebraically closed field $K$, we mean that if $H$ has pinned root datum $(X, \Pi, X^{\vee}, \Pi^{\vee}, \{x_{\alpha_i}\})$, then $\sigma$ preserves $T \subset B \subset H$, and $\sigma \circ x_{\alpha_i} = x_{\sigma(\alpha_i)}$ for each $\alpha_i \in \Pi$. Then $\sigma$ also acts on $X$, $\Pi$, $X^{\vee}$, and $\Pi^{\vee}$, preserving Dynkin diagram edges and abelian group structure.
   
   \begin{Prop}
    \label{Prop:root-data}
    Let $G$ be a complex, connected, reductive group with pinning $T \subset B \subset G$ and $\{x_{\alpha_i}\}_{\alpha_i \in \Pi}$. Let the corresponding pinned root datum be denoted $(X, \Pi, X^{\vee}, \Pi^{\vee}, \{x_{\alpha_i}\})$. Let $\sigma$ be an automorphism of $G$ preserving its pinning. Then the fixed point subgroup $G^{\sigma} \subset G$ is a closed subgroup, and is reductive. The neutral component $G^{\sigma, \circ} \subset G^{\sigma} \subset G$ is also a closed subgroup and a connected, reductive group. The root datum of $G^{\sigma, \circ}$ is $(X_{\sigma}/tor, \res{\sigma}{\Pi}, (X^{\vee})^{\sigma}, N'_{\sigma}(\Pi^{\vee}), \{x_{\alpha_{\eta}}\})$, where
    \begin{enumerate}[label=\roman*.]
     \item
      \label{Prop:root-data:i}
      $X_{\sigma}$ is the group of $\sigma$-coinvariants $X / \langle x - \sigma(x) \rangle$. $X_{\sigma}/tor$ is the quotient by all torsion elements.
     
     \item
      \label{Prop:root-data:ii}
      For every orbit of simple roots $\eta \subset \Pi$, there is a single root $\alpha_{\eta}$ equal to the image of any $\alpha_i \in \eta$ under the quotient map $X \to X_{\sigma}/tor$. Then $\res{\sigma}{\Pi} = \{\alpha_{\eta} \mid \eta \in \Pi/\sigma \}$.
      
     \item
      \label{Prop:root-data:iii}
      $(X^{\vee})^{\sigma} \subset X^{\vee}$ is the subgroup of $\sigma$-invariant cocharacters.
      
     \item
      \label{Prop:root-data:iv}
      For every $\sigma$-orbit $\eta \subset \Pi^{\vee}$, we have $\alpha^{\vee}_{\eta} = \sum_{\alpha^{\vee}_i \in \eta} \alpha^{\vee}_i$ in the case $\eta$ consists of pairwise disconnected simple roots, and $\alpha^{\vee}_{\eta} = 2(\sum_{\alpha^{\vee}_i \in \eta} \alpha^{\vee}_i)$ in the case $\eta$ consists of a pair of simple roots connected by an edge. Then $N'_{\sigma}(\Pi^{\vee}) = \{\alpha^{\vee}_{\eta} \mid \eta \in \Pi^{\vee}/\sigma \}$.
    \end{enumerate}
   \end{Prop}
   
   \begin{proof}
    See \cite{St68} chapters 7--8. See also \cite{Ha15} and \cite{Ha18}.
   \end{proof}
   
   Note that if $\pi: X \to X_{\sigma}/tor$ and we have a character $\lambda \in X_{\sigma}/tor$ and cocharacter $\mu \in (X^{\vee})^{\sigma}$, then $\langle \lambda, \mu \rangle$ is given by $\langle \tilde{\lambda}, \mu \rangle$, where $\tilde{\lambda}$ is any lift $\pi(\tilde{\lambda}) = \lambda$.
   
   \begin{Prop}
    \label{Prop:Weyl}
    The group of $\sigma$-invariants $W^{\sigma}$ is the Coxeter group generated by $s_{\eta}$, where for each orbit $\eta$ of simple roots, $s_{\eta}$ is the longest element of the group generated by simple reflections in $\eta$.
    
    Furthermore, $W^{\sigma}$ is the Weyl group for $G^{\sigma, \circ}$, i.e. $W^{\sigma} = N_G(T^{\sigma, \circ})/T^{\sigma, \circ}$.
   \end{Prop}
   
   \begin{proof}
    See \cite{St68} and \cite{Ha15}.
   \end{proof}
   
   We define $\res{\sigma}{\Phi}$ and $N'_{\sigma}(\Phi^{\vee})$ for the root system and coroot system of $G^{\sigma, \circ}$ as follows:
   \[
    \res{\sigma}{\Phi} := W^{\sigma} \cdot \res{\sigma}{\Pi}, \qquad
    \text{and} \qquad
    N'_{\sigma}(\Phi^{\vee}) := W^{\sigma} \cdot N'_{\sigma}(\Pi^{\vee}).
   \]
   
   \begin{Prop}
    \label{Prop:rho}
    The half sum of positive coroots is equal for $G$ and $G^{\sigma, \circ}$: $\rho^{\vee} = \rho^{\vee}_{\sigma}$.
   \end{Prop}
   
   \begin{proof}
    Recall $\rho^{\vee}$ is the sum of fundamental coweights, or dual vectors to the simple roots under the natural perfect pairing $\langle{\cdot,\cdot}\rangle$. For $\alpha_i \in \Pi$, let $\lambda^i$ be the corresponding fundamental coweight. Similarly, for $\alpha_{\eta} \in \res{\sigma}{\Pi}$, let $\lambda^{\eta}$ be the corresponding fundamental coweight. It is sufficient to show that for each $\eta$,
    \[
     \lambda^{\eta} =
     \sum_{i \in \eta} \lambda^i.
    \]
    
    For any two orbits $\zeta$ and $\eta$,
    \[
     \langle \alpha_{\zeta}, \sum_{i \in \eta} \lambda^i \rangle =
     \sum_{i \in \eta} \langle \alpha_j, \lambda^i \rangle =
     \delta_{\zeta, \eta},
    \]
    where $\alpha_j$ is any of the simple roots mapping to $\alpha_{\zeta}$. This works regardless of type of orbits, since $\mathrm{res}_{\sigma}$ treats all simple roots uniformly.
   \end{proof}
   
   Note that $\rho_{\sigma} \ne \rho$. In particular, if $\lambda$ is a $\sigma$-invariant cocharacter, then in general $|\langle \rho_{\sigma}, \lambda \rangle| \le |\langle \rho, \lambda \rangle|$. As a result, by Theorem \ref{THM:MV} \ref{THM:MV:i}, $\dimp{A^{\sigma}} \le \dim{A}$ for a $\sigma$-invariant AMV cycle $A$.
   
   Now we can prove the following theorem, which appears in \cite{Ha18} as Theorem 7.7, although it is not proved in full generality there.
   
   \begin{THM}
    \label{H-7.7}
    Let $G$ be a quasi-split, connected reductive group over a non-Archimedean local field $F$ with inertia group $I$ and geometric Frobenius $\tau$. Let $\widehat{G}$ be the complex dual of $G$, and let $\widehat{G}^I$ be the fixed point subgroup of $\widehat{G}$, and $\widehat{G}^{I, \circ}$ the neutral component. Let the root system of $\widehat{G}^{I, \circ}$ be denoted $\breve{\Sigma}^{\vee}$. Then $\tau$ is an outer automorphism of $\widehat{G}^{I, \circ}$ preserving the natural pinning. Let $V_{\lambda, \xi}$ be the highest-weight representation of $\widehat{G}^{I, \circ} \rtimes \langle \tau \rangle$ where $\tau$ acts by the scalar $\xi \in \CC^{\times}$ on weight spaces associated to weights $\nu \in W^{\tau} \cdot \lambda$.
    
    Let $\lambda \in X^*(\widehat{T}^{I, \circ})^{+, \tau}$ be a dominant, $\tau$-invariant character of $\widehat{T}^{I, \circ}$. There is an equality
    
    \[
     \sum_{\nu \in Wt(\lambda)^{\tau}} \trres{\tau}{V_{\lambda, 1} (\nu)} e^{\nu} =
     \sum_{w \in W^{\tau}} w \left(\prod_{\alpha \in N'_{\tau} (\breve{\Sigma}^{\vee})^+} \frac{1}{1 - e^{-\alpha}} \right) e^{w(\lambda)}.
    \]
   \end{THM}
   
   \begin{proof}
    The group $\widehat{G}^{I, \circ}$ is connected and reductive. And since $G$ is quasi-split over $F$, $\tau$ acts on $\widehat{G}^{I, \circ}$, preserving the natural pinning. Then the theorem follows from Theorem \ref{THM:main} by two observations. First, $V_{\lambda, 1} = V_{\lambda}$ as a vector space and carries the same normalized $\tau$-action as that described in Section \ref{S:Eigenvalues}.
    
    And second, $N'_{\tau}(\breve{\Sigma}^{\vee})$ is the set of roots of $\widehat{\widehat{\widehat{G}^{I, \circ}}^{\tau, \circ}}$. Indeed, $\breve{\Sigma}^{\vee}$ is the set of roots of $\widehat{G}^{I, \circ}$, so it is the set of coroots of the complex dual group $\widehat{\widehat{G}^{I, \circ}}$. Then by Proposition \ref{Prop:root-data} \ref{Prop:root-data:iv}, $N'_{\tau}(\breve{\Sigma}^{\vee})$ is the set of coroots of the neutral component of the $\tau$-fixed subgroup $\widehat{\widehat{G}^{I, \circ}}^{\tau, \circ} \subset \widehat{\widehat{G}^{I, \circ}}$. And finally, by again taking complex dual, $N'_{\tau}(\breve{\Sigma}^{\vee})$ is the set of roots of $\widehat{\widehat{\widehat{G}^{I, \circ}}^{\tau, \circ}}$.
   \end{proof}
   
   Note that in Theorem \ref{H-7.7}, $G$ is assumed to be quasi-split over a non-Archimedean local field. In particular, $G$ may be ramified. Then the root system $\breve \Sigma^{\vee}$ for $\widehat{G}^{I, \circ}$ may be determined combinatorially from the absolute root datum of $G$, along with the action of the Galois group. In particular, as shown in \cite{Ha18} Theorem 6.8, $N'_\tau(\breve \Sigma^\vee)$ is equal to the root system $\widetilde\Sigma_0^\vee$ appearing in the Lusztig character formula of \cite{Kn05}. This is a necessary ingredient in the proof of Theorem D in \cite{Ha18}.

 \appendix
 
 \section{Geometric Satake equivalence}
  \label{app}
  I use several theorems of \cite{MV07}, summarized in Theorem \ref{THM:MV}. In wording more similar to that used by Mirkovi\'c and Vilonen, along with numbers of specific statements, I have the following:
  
  \begin{THM}[\cite{MV07} Theorem 3.2]
   \label{THM:MV:3.2}
   \begin{enumerate}[label=\alph*)]
    \item 
     The intersection $S_e^{\nu} \cap \GrG{G}^{\mu}$ is nonempty precisely when $\varpi^{\nu} \in \overline{\GrG{G}^{\mu}}$ and then $S_e^{\nu} \cap \overline{\GrG{G}^{\mu}}$ is of pure dimension $\langle{\rho,\mu + \nu}\rangle$, if $\mu$ is chosen dominant.
     
    \item
     The intersection $S_{w_0}^{\nu} \cap \GrG{G}^{\mu}$ is nonempty precisely when $\varpi^{\nu} \in \overline{\GrG{G}^{\mu}}$ and then $S_{w_0}^{\nu} \cap \overline{\GrG{G}^{\mu}}$ is of pure dimension $-\langle{\rho, \mu + \nu}\rangle$, if $\mu$ is chosen anti-dominant.
   \end{enumerate}
  \end{THM}
  
  \begin{THM}[\cite{MV07} Theroem 3.5]
   \label{THM:MV:3.5}
   For all $\Aa \in P_{L^+G} (\GrG{G}, K)$ there is a canonical isomorphism
   \[
    H_c^k (S_e^{\nu}, \Aa) \overset{\sim}{\to}
    H_{S_{w_0}^{\nu}}^k (\GrG{G}, \Aa)
   \]
   and both sides vanish for $k \ne 2\langle{\rho, \nu}\rangle$.
   
   In particular, the functors $F_{\nu}: P_{L^+G}(\GrG{G}, K) \to \ModA{K}$, defined by
   \[
    F_{\nu} :=
    H_c^{2\langle{\rho, \nu}\rangle} (S_e^{\nu}, -) =
    H_{S_{w_0}^{\nu}}^{2\langle{\rho, \nu}\rangle} (\GrG{G}, -),
   \]
   are exact.
  \end{THM}
  
  \begin{THM}[\cite{MV07} Theorem 3.6]
   \label{THM:MV:3.6}
   \[
    \HH^{\bullet} \cong
    \bigoplus_{\nu \in X_*(T)} F_{\nu} =
    \bigoplus_{\nu \in X_*(T)} H_c^{2 \langle{\rho, \nu}\rangle} (S_e^{\nu}, -) :
    P_{L^+G} (\GrG{G}, K) \to \Vect{K}
   \]
  \end{THM}
  
  \begin{Prop}[\cite{MV07} Proposition 3.10]
   \label{Prop:MV:3.10}
   Let $R$ be a Noetherian ring of finite global dimension. There is a canonical identification
   \[
    H_c^{2\langle{\rho, \lambda}\rangle} (S_e^{\lambda}, IC_{\mu}(R)) \cong
    H_c^{2\langle{\rho, \mu - \lambda}\rangle} (S_e^{\lambda} \cap \GrG{G}^{\mu}, R) \cong
    R[\Irr{S_e^{\lambda} \cap \overline{\GrG{G}^{\mu}}}],
   \]
   here $R[\Irr{S_e^{\lambda} \cap \overline{\GrG{G}^{\mu}}}]$ stands for the free $R$-module generated by the irreducible components of $S_e^{\lambda} \cap \overline{\GrG{G}^{\mu}}$.
  \end{Prop}
  
  Note that \cite{MV07} Theorem 3.10 is proved using a constant sheaf on the smooth variety $S_e^{\lambda} \cap \GrG{G}^{\mu}$, as written above. The second isomorphism thus follows from the bijection between $\Irr{S_e^{\lambda} \cap \GrG{G}^{\mu}}$ and $\Irr{S_e^{\lambda} \cap \overline{\GrG{G}^{\mu}}}$.
  
  \begin{THM}[\cite{MV07} Theorem 12.1]
   \label{THM:MV:main}
   The group scheme $\widetilde{G}_{\ZZ}$ is the split reductive group scheme over $\ZZ$ whose root datum is dual to that of $G$.
  \end{THM}
  
  Here $\widetilde{G}_{\ZZ}$ is the group over $\ZZ$ whose category of representations is isomorphic as a tensor category to $P_{L^+G} (\GrG{G}, \ZZ)$. Mirkovi\'c and Vilonen construct $\widetilde{G}_{\ZZ}$ as a $\ZZ$-scheme so a result analogous to Theorem \ref{THM:MV:main} will hold, by base change, for coefficients in any Noetherian ring $R$ of finite global dimension. In particular, the complex group with dual root datum is $\widehat{G}$, and so $\widehat{G} \cong \widetilde{G} := \Spec{\CC} \times_{\ZZ} \widetilde{G}_{\ZZ}$ has a representation category isomorphic to $P_{L^+G}(\GrG{G}, \CC)$. Furthermore, $\widehat{G^{\sigma, \circ}} \cong \widetilde{G^{\sigma, \circ}} := \Spec{\CC} \times_{\ZZ} \widetilde{G^{\sigma, \circ}}_{\ZZ}$. 
  
  Note that $\widetilde{G}$ is naturally endowed with a maximal torus and Borel $\widetilde{T} \subset \widetilde{B} \subset \widetilde{G}$, identifiable using representations of $\widetilde{G}$. In particular, consider the representation 
  \[
   \tilde{\frakg}^{ss} :=
   Lie([\widetilde{G}, \widetilde{G}]) = 
   \bigoplus_i \HH^{\bullet}(\GrG{G}, IC_{\gamma_i^{\vee}}),
  \]
  where $i$ runs through the components of the Dynkin diagram of $G$, and $\gamma_i^{\vee}$ is the highest coroot in $\Phi^{\vee}_i$. Then we have a decomposition into weight spaces
  \[
   \tilde{\frakg}^{ss} =
   \bigoplus_i \left(H_c^0 (S_{w_0}^0, IC_{\gamma^{\vee}_i}) \oplus \bigoplus_{\alpha^{\vee} \in \Phi^{\vee}_i} H_c^{-2\langle{\rho,\alpha^{\vee}}\rangle} (S_{w_0}^{\alpha^{\vee}}, IC_{\gamma^{\vee}_i})\right) =
   \tilde{\frakg}^{ss} (0) \oplus \bigoplus_{\alpha^{\vee} \in \Phi^{\vee}} \tilde{\frakg}^{ss} (\alpha^{\vee}).
  \]
  For $g \in \widetilde{G}$, we can say $g \in \widetilde{T}$ if $g$ preserves all weight spaces of $\tilde{\frakg}^{ss}$, and $g \in \widetilde{B}$ if $g$ preserves the vector subspace
  \[
   \tilde{\frakg}^{ss} (0) \oplus \bigoplus_{\alpha^{\vee} \in \Phi^{\vee, +}} \tilde{\frakg}^{ss} (\alpha^{\vee}).
  \]
  Thus if we fix pinnings of $\widetilde{G}$ and $\widehat{G}$, we can identify the two groups uniquely.

  \begin{THM}[\cite{MV07} Corollary 13.2]
   \label{Cor:MV:13.2}
   Let $R$ be a Noetherian ring of finite global dimension. The $\lambda$-weight spaces $S_{\mu}(\lambda)$ and $W_{\mu} (\lambda)$ of $S_{\mu}$ and $W_{\mu}$, respectively, can both be canonically identified with the free $R$-module spanned by the irreducible components of $S_e^{\lambda} \cap \overline{\GrG{G}^{\mu}}$. In particular, the ranks of these modules can be given by the number of irreducible components of $S_e^{\lambda} \cap \overline{\GrG{G}^{\mu}}$.
  \end{THM}
  
  In the corollary above, $S_{\mu}$ and $W_{\mu}$ are canonical $R$-representations of $\widehat{G}$. In particular, taking coefficients in $R = \CC$ (or any other field), we have a natural map $S_{\mu} \to W_{\mu}$, bijective on underlying vector spaces, factoring through the irreducible highest weight representation $V_{\mu}$. As a result, the underlying vector space of $V_{\mu}$ has a basis indexed by $\Irr{S_e^{\lambda} \cap \overline{\GrG{G}^{\lambda}}}$.
  
  \begin{Lemma}[\cite{BG08} Proposition 5 (iii)]
   \label{Lemma:MV=AMV}
   Let $\nu \in X_*(T)$ be such that $\nu \ge 0$. If $\mu \in X_*(T)^+$ is sufficiently dominant, then $S_{w_0}^{\mu - \nu} \cap S_e^{\mu} = S_{w_0}^{\mu - \nu} \cap \GrG{G}^{\mu}$.
  \end{Lemma}
  
  Here a dominant cocharacter $\mu$ may be considered ``sufficiently dominant'' if, for all simple roots $\alpha \in \Pi$, we have $\langle{\alpha, \mu}\rangle \ge N$, where $N$ is some positive integer depending on the group $G$ and the cocharacter $\nu$.
  
  These are all the statements necessary to state and prove Theorem \ref{THM:MV}:
  
  \begin{THM}
    \label{app:THM:MV}
    Let $\mu$ be a dominant cocharacter, and let $\lambda \in Wt(\mu)$.
    \begin{enumerate}[label=\roman*.]
     \item
      \label{app:THM:MV:i}
      $S_{w_0}^{\lambda} \cap \overline{\GrG{G}^{\mu}}$ is equidimensional, and $\dimp{S_{w_0}^{\lambda} \cap \overline{\GrG{G}^{\mu}}} = \langle \rho, \mu - \lambda \rangle$
      
     \item
      \label{app:THM:MV:ii}
      $S_{w_0}^{\lambda} \cap S_e^{\mu}$ is equidimensional, and $\dimp{S_{w_0}^{\lambda} \cap S_e^{\mu}} = \langle \rho, \mu - \lambda \rangle$
     
     \item
      \label{app:THM:MV:iii}
      $\displaystyle \HH^{\bullet}(\GrG{G}, IC_{\mu}) = \bigoplus_{\lambda \in Wt(\mu)} H_c^{-2\langle \rho, \lambda\rangle} (S_{w_0}^{\lambda}, IC_{\mu}) =
       V_{\mu}$
      
     \item
      \label{app:THM:MV:iv}
      $\displaystyle H_c^{-2\langle \rho, \lambda\rangle}(S_{w_0}^{\lambda}, IC_{\mu}) =
       \bigoplus_{A \in \Irr{S_{w_0}^{\lambda} \cap \overline{\GrG{G}^{\mu}}}} \CC[A] =
       V_{\mu}(\lambda)$.
       
      \item
       \label{app:THM:MV:v}
       $P_{L^+G} (\GrG{G}, \ZZ)$ is isomorphic as a tensor category to $\Rep{\ZZ}{\widehat{G}}$. 
     \end{enumerate}
   \end{THM}
   
   \begin{proof}
    Statement \ref{app:THM:MV:i} follows immediately from \ref{THM:MV:3.2} (b). Note that $-\mu$ is anti-dominant exactly when $\mu$ is dominant.
    
    Statement \ref{app:THM:MV:ii} is not a direct consequence of any statement in \cite{MV07}, but it is well-known. One way to see it follows from Lemma \ref{Lemma:MV=AMV} and statement \ref{app:THM:MV:i}. Indeed, note that for $\nu \in X_*(T)$, the translation of semi-infinite cells $\nu: S_w^{\eta} \to S_w^{\eta + \nu}$ is an isomorphism of ind-schemes. In particular,
    \[
     \dimp{S_{w_0}^{\lambda} \cap S_e^{\mu}} =
     \dimp{S_{w_0}^{\lambda + \nu} \cap S_e^{\mu + \nu}}.
    \]
    So $\dimp{S_{w_0}^{\lambda} \cap S_e^{\mu}} = \dimp{S_{w_0}^{\lambda + n \rho^{\vee}} \cap S_e^{\mu + n \rho^{\vee}}}$ for all integers $n$. For sufficiently large $n$, the character $\mu + n \rho^{\vee}$ is sufficiently dominant, with respect to $G$ and $\mu - \lambda$, to satisfy the hypotheses of Lemma \ref{Lemma:MV=AMV}. Therefore
    \[
     \dimp{S_{w_0}^{\lambda} \cap S_e^{\mu}} =
     \dimp{S_{w_0}^{\lambda + n \rho^{\vee}} \cap S_e^{\mu + n \rho^{\vee}}} =
     \dimp{S_{w_0}^{\lambda + n \rho^{\vee}} \cap \GrG{G}^{\mu+ n \rho^{\vee}}} =
     \langle{\rho, \mu - \lambda}\rangle.
    \]
    
    Statements \ref{app:THM:MV:iii} and \ref{app:THM:MV:iv} are closely tied together. Theorem \ref{THM:MV:3.5} tells us that the global cohomology of $L^+G$-equivariant perverse sheaves decomposes canonically as a direct sum of cohomology groups with compact support, taken on semi-infinite cells. Specifically, for $\Aa \in P_{L^+G} (\GrG{G})$,
    \[
     \HH^{\bullet}(\GrG{G}, \Aa) =
     \bigoplus_{\nu \in X_*(T)} H_c ^{2 \langle{\rho, \nu}\rangle} (S_e^{\nu}, \Aa).
    \]
    By symmetry between our choice of Borel and its opposite, we have
    \begin{equation}
     \label{eq:MV:negsign}
     H_c^{2\langle{\rho, \nu}\rangle} (S_e^{\nu}, IC_{\mu}) =
     H_c^{-2\langle{\rho, \nu}\rangle} (S_{w_0}^{\nu}, IC_{\mu}).
    \end{equation}
    And of course by emptiness of the intersection of $S_{w_0}^{\lambda} \cap \GrG{G}^{\mu}$ for $\lambda \not\in Wt(\mu)$, the only $\lambda$ appearing in the direct sum for $\HH^{\bullet}$ are those contained in $Wt(\mu)$. Thus the first equality of \ref{THM:MV:iii} follows. And the first equality of \ref{THM:MV:iv} follows from Proposition \ref{Prop:MV:3.10}.
    
    The second equality of \ref{THM:MV:iv} follows from Theorem \ref{Cor:MV:13.2} and equation (\ref{eq:MV:negsign}), which in turn implies the second equality of \ref{THM:MV:iii}.
    
    Statement \ref{app:THM:MV:v} is a restatement of Theorem \ref{THM:MV:main}.
   \end{proof}
   
  \printbibliography
\end{document}